%% file: arXiv_SVRC.tex
\documentclass[11pt]{article}
\pdfoutput=1
\usepackage{enumerate}
\usepackage{pdfsync}
\usepackage[OT1]{fontenc}

\usepackage[usenames]{color}
\usepackage{smile}
\usepackage[colorlinks,
            linkcolor=red,
            anchorcolor=blue,
            citecolor=blue
            ]{hyperref}
\usepackage{fullpage}
\usepackage{hyperref}
\usepackage[protrusion=true,expansion=true]{microtype}
\usepackage{pbox}
\usepackage{setspace}
\usepackage{tabularx}
\usepackage{float}
\usepackage{wrapfig,lipsum}
\usepackage{enumitem}

\makeatletter
\newcommand*{\rom}[1]{\expandafter\@slowromancap\romannumeral #1@}
\makeatother
\allowdisplaybreaks

\newcommand{\la}{\langle}
\newcommand{\ra}{\rangle}

\def \vst {\vb_t^s}
\def \Ust {\Ub_t^s}
\def \Mst {M_{s,t}}

\def \hst {\hb_t^s}
\def \bst {b_{s,t}}
\def \Bst {B_{s,t}}
\def \alphast {\alpha_{s,t}}
\def \betast {\beta_{s,t}}
\def \rx {\hat{\xb}^{s-1}}
\def \rg {\gb^s}
\def \rH {\Hb^s}
\def \xold {\xb_t^s}
\def \xnew {\xb_{t+1}^s}
\def \xout {\xb_{\text{out}}}
\def \constant {C}
\def \bg {b}
\def \bH {B}
\def \df {\nabla f}
\def \Hf {\nabla^2 f}
\def \dF {\nabla F}
\def \HF {\nabla^2 F}
\def \ev {\eb_{\vb}}
\def \eU {\eb_{\Ub}}
\def \cold {c_{s,t}}
\def \cnew {c_{s,t+1}}
\def \Gst {\Gamma_{s,t}}
\def \minG {\gamma_n}
\def \Gt {\Gamma_t}
\def \ct {c_t}
\def \ctnew {c_{t+1}}
\def \glip {L_1}
\def \Hlip {L_2}
\def \mineig {\lambda_{\text{min}}}
\def \bigO {O}
\def \tbigO {\tilde{\bigO}}
\def \CR {\text{CR}}
\def \SCR {\text{SCR}}
\def \SVRC {\text{SVRC}}
\def \SVRCp {\text{Lite-SVRC}}

\begin{document}
\title{\huge Sample Efficient Stochastic Variance-Reduced Cubic Regularization Method}
\author
{
	Dongruo Zhou\thanks{Department of Computer Science, University of California, Los Angeles, CA 90095, USA; e-mail: {\tt drzhou@cs.ucla.edu}} 
	~~~and~~~
	Pan Xu\thanks{Department of Computer Science, University of California, Los Angeles, CA 90095, USA; e-mail: {\tt panxu@cs.ucla.edu}} 
	~~~and~~~
	Quanquan Gu\thanks{Department of Computer Science, University of California, Los Angeles, CA 90095, USA; e-mail: {\tt qgu@cs.ucla.edu}}
}
\date{May 18, 2018\footnote{The first version of this paper was submitted to UAI 2018 on March 9, 2018. This is the second version with improved presentation and additional baselines in the experiments, and was submitted to NeurIPS 2018 on May 18, 2018. }}

\maketitle

\begin{abstract}
We propose a sample efficient stochastic variance-reduced cubic regularization (Lite-SVRC) algorithm for finding the local minimum efficiently in nonconvex optimization. The proposed algorithm achieves a lower sample complexity of Hessian matrix computation than existing cubic regularization based methods. At the heart of our analysis is the choice of a constant batch size of Hessian matrix computation at each iteration and the stochastic variance reduction techniques. In detail, for a nonconvex function with $n$ component functions, Lite-SVRC converges to the local minimum within $\tilde{O}(n+n^{2/3}/\epsilon^{3/2})$\footnote{Here $\tilde{O}$ hides poly-logarithmic factors} Hessian sample complexity, which is faster than all existing cubic regularization based methods. Numerical experiments with different nonconvex optimization problems conducted on real datasets validate our theoretical results. 
\end{abstract}

\input{intro}

\input{Algorithm}

\input{theory}

\input{Experiments}

\input{Conclusion}

\appendix
\input{appendix.tex}

\bibliographystyle{ims}
\bibliography{reference}

\end{document}

%% file: intro.tex
\section{Introduction}

We study the following unconstrained finite-sum nonconvex optimization problem:
\begin{align}\label{def:problem}
    \min_{\xb \in \RR^d} F(\xb) = \frac{1}{n}\sum_{i=1}^n f_i(\xb), 
\end{align}
where each $f_i:\RR^d\rightarrow \RR$ is a general nonconvex function. Such nonconvex optimization problems are ubiquitous in machine learning, including training deep neural network \citep{lecun2015deep}, robust linear regression \citep{yu2017robust} and nonconvex regularized logistic regression \citep{reddi2016fast}. In principle, finding the global minimum of \eqref{def:problem} is generally a NP-hard problem \citep{hillar2013most} due to the lack of convexity. 

Instead of finding the global minimum, various algorithms have been developed in the literature \citep{Nesterov2006Cubic, Cartis2011Adaptive, Carmon2016Gradient,Agarwal2017Finding, xu2017neon+, AllenLi2017-neon2} to find an approximate local minimum of \eqref{def:problem}. In particular, a point $\xb$ is said to be an $(\epsilon_g, \epsilon_H)$-approximate local minimum of $F$ if 
\begin{align}\label{eq:def_localminima}
    \|\dF(\xb)\|_2\leq \epsilon_g, \quad\mineig(\HF(\xb)) \geq -\epsilon_H,
\end{align}
where $\epsilon_g,\epsilon_H>0$ are predefined precision parameters. It has been shown that such approximate local minima can be as good as global minima in some problems. For instance, \citet{ge2016matrix} proved that any local minimum is actually a global minimum in matrix completion problems. 
Therefore, to develop an algorithm to find an approximate local minimum is of great interest both in theory and practice. 

A very important and popular method to find the approximate local minimum is cubic-regularized (CR) Newton method, which was originally introduced by \cite{Nesterov2006Cubic}. Generally speaking, in the $k$-th iteration, CR solves a sub-problem which minimizes a cubic-regularized second-order Taylor expansion at current iterate $\xb_k$. The update rule can be written as follows:
\begin{align}
    \hb_k &= \argmin_{\hb \in \RR^d} \la \dF(\xb_k), \hb\ra  +\frac{1}{2}\la \HF(\xb_k)\hb, \hb\ra+\frac{M}{6}\|\hb\|_2^3,\label{def:total_sub_problem}\\
    \xb_{k+1}& = \xb_k+\hb_k, \label{def:total_sub_problem2}
\end{align}
where $M>0$ is a penalty parameter used in CR. \citet{Nesterov2006Cubic} proved that to find an $(\epsilon, \sqrt{\epsilon})$-approximate local minimum of a nonconvex function $F$, CR requires at most $\bigO(\epsilon^{-3/2})$ iterations. However, a main drawback for CR is that it needs to sample $n$ individual Hessian matrix $\Hf_i(\xb_k)$ to get the exact Hessian $\nabla^2 F(\xb_k)$ used in \eqref{def:total_sub_problem}, which leads to a total $\bigO(n\epsilon^{-3/2})$ Hessian sample complexity, i.e., number of queries to the stochastic Hessian $\nabla^2 f_i(\xb)$ for some $i$ and $\xb$. Such computational cost will be extremely expensive when $n$ is large as is in many large scale machine learning problems. 


To overcome the computational burden of CR based methods, some recent studies have proposed to use sub-sampled Hessian instead of the full Hessian \citep{kohler2017sub, xu2017second} 
to reduce the Hessian complexity. In detail, \citet{kohler2017sub} proposed a sub-sampled cubic-regularized Newton method ($\SCR$), which uses a subsampled Hessian instead of full Hessian to reduce the per iteration sample complexity of Hessian evaluations. 
\citet{xu2017second} proposed a refined convergence analysis of $\SCR$, as well as a subsampled Trust Region algorithm \citep{conn2000trust}. 
Nevertheless, $\SCR$ bears a much slower convergence rate than the original CR method, 
and the total Hessian sample complexity for $\SCR$ to achieve an $(\epsilon, \sqrt{\epsilon})$-approximate local minimum is $\tbigO(\epsilon^{-5/2})$. This suggests that the computational cost of $\SCR$ could be even worse than $\CR$ when $\epsilon  \lesssim n^{-1}$.

In order to retain the fast convergence rate of CR and enjoy the computational efficiency of $\SCR$, \citet{zhou2018stochastic} proposed a stochastic variance-reduced cubic-regularized Newton methods ($\SVRC$) to further improve the convergence rate of stochastic CR method. 
At the core of SVRC is an innovative semi-stochastic gradient, as well as a semi-stochastic Hessian \citep{gower2017tracking,wai2017curvature}. 
They proved that SVRC achieves an $(\epsilon, \sqrt{\epsilon})$-approximate local minimum with $\tbigO(n+n^{4/5} \epsilon^{-3/2})$ second-order oracle complexity, which is defined to be the number of queries to the second-order oracle, i.e., a triplet $(f_i(\xb), \nabla f_i(\xb), \nabla^2 f_i(\xb))$.
However, the second-order oracle complexity is dominated by the maximum number of queries to one of the elements in the triplet triplet $(f_i(\xb), \nabla f_i(\xb), \nabla^2 f_i(\xb))$, and therefore is not always accurate in reflecting the true computational complexity. 
For instance, Algorithm A with higher second-order oracle complexity may be due to its need to query more stochastic gradients ($\nabla f_i(\xb)$'s) than Algorithm B, but it may need to query much fewer stochastic Hessians ($\nabla^2 f_i(\xb))$'s) than Algorithm B. Given that the computational complexity of the stochastic Hessian matrix is $\bigO(d^2)$  while that of the stochastic gradient is only $\bigO(d)$, Algorithm B can be more efficient than Algorithm A.  
In other words, an algorithm with higher second-order oracle complexity is not necessarily slower than the other algorithm with lower second-order oracle complexity. 
So it is more reasonable to use the Hessian sample complexity to evaluate the efficiency of cubic regularization methods when the dimension $d$ is not small. Recently, \cite{wang2018sample} proposed another variance reduced stochastic cubic regularization algorithm, 
which achieves $\tbigO(n+n^{8/11} \epsilon^{-3/2})$ Hessian sample complexity\footnote{They actually missed additional $O(n)$ Hessian sample complexity since their algorithms need to calculate the minimum eigenvalue of Hessian as a stopping criteria in each iteration.} to converge to an $(\epsilon, \sqrt{\epsilon})$-approximate local minimum. 



In this paper, in order to reduce the Hessian sample complexity, we develop a sample efficient stochastic variance-reduced cubic-regularized Newton method called $\SVRCp$, which significantly reduces the sample complexity of Hessian matrix evaluations in stochastic CR methods. In detail, under milder conditions, we prove that $\SVRCp$ achieves a lower Hessian sample complexity than existing cubic regularization based methods. Numerical experiments with different types of nonconvex optimization problems on various real datasets are conducted to validate our theoretical results. 


We summarize our major contributions as follows:
\begin{itemize}[leftmargin=*]
    \item The proposed Lite-SVRC algorithm only requires a constant batch size of Hessian evaluations at each iteration. In contrast, the batch size of Hessian evaluations at each iteration in \cite{wang2018sample} is implicitly chosen based on the update of the next iterate. 
    \item We prove that $\SVRCp$ converges to an $(\epsilon, \sqrt{\epsilon})$-approximate local minimum of a nonconvex function within $\tbigO(n+n^{2/3}\epsilon^{-3/2})$ Hessian sample complexity, which outperforms all the state-of-the-art cubic regularization algorithms including \cite{zhou2018stochastic,wang2018sample}. 
    \item Last but not the least, our results do not require the Lipschitz continuous condition of $F(\xb)$, which directly improves the results in \cite{wang2018sample} that rely on this additional assumption. 
\end{itemize}

\input{related}

\noindent\textbf{Notation:}  We use $a(x) = \bigO(b(x))$ if $a(x) \leq Cb(x)$, where $C$ is a constant independent of any parameters in our algorithm. We use $\tilde O(\cdot)$ to hide polynomial logarithm terms. We use $\|\vb\|_2$ to denote the 2-norm of vector $\vb \in \RR^d$. For symmetric matrix $\Hb \in \RR^{d \times d}$, we use $\|\Hb\|_2$ and $\|\Hb\|_{S_r}$ to denote the spectral norm and Schatten $r$- norm of $\Hb$. We denote the smallest eigenvalue of $\Hb$ to be $\mineig(\Hb)$. 

%% file: related.tex
\subsection{Additional Related Work}


\noindent\textbf{Cubic Regularization and Trust-region Newton Method}
Traditional Newton method in convex setting has been widely studied in past decades \citep{bennett1916newton, bertsekas1999nonlinear}. In the nonconvex setting, based upon cubic-regularized Newton method \citep{Nesterov2006Cubic}, 
\citet{Cartis2011Adaptive} proposed a practical framework of cubic regularization which uses an adaptive cubic penalty parameter and approximate cubic sub-problem solver. \citet{Carmon2016Gradient, Agarwal2017Finding} presented two fast cubic-regularized methods where they used only gradient and Hessian-vector product to solve the cubic sub-problem. 
\citet{tripuraneni2017stochastic} developed a stochastic cubic regularization algorithm based on \cite{kohler2017sub} where only gradient and Hessian vector product are used. 
The other line of related research is trust-region Newton methods \citep{conn2000trust,carrizo2016trust, curtis2017trust,CurtScheShi17}, which have comparable performance guarantees as cubic regularization methods. 

\noindent\textbf{Finding Approximate Local Minima}
There is another line of work which focuses on finding approximate local minima using the negative curvature. \citet{ge2015escaping, Jin2017How} showed that (stochastic) gradient descent with an injected uniform noise over a small ball is able to converge to approximate local minima. \citet{Carmon2016Accelerated, royer2017complexity, allen2017natasha} showed that one can find approximate local minima faster than first-order methods by using Hessian vector product to extract information of negative curvature. \citet{xu2017neon+, AllenLi2017-neon2, jin2017accelerated} further proved that gradient methods with bounded perturbation noise are also able to find approximate local minima faster than the first-order methods.

\noindent\textbf{Variance Reduction}
Variance-reduced techniques play an important role in our proposed algorithm. \citet{roux2012stochastic, johnson2013accelerating} proved that stochastic gradient descent(SGD) with variance reduction is able to converge to global minimum much faster than SGD in convex setting. In the nonconvex setting, \citet{Reddi2016Stochastic, allen2016variance} show that stochastic variance-reduced gradient descent (SVRG) is able to converge to first-order stationary point with the same convergence rate as gradient descent, yet with an $\Omega(n^{1/3})$ improvement in gradient complexity. 

The remainder of this paper is organized as follows: we present our proposed algorithm in Section \ref{sec:alg}. In Section \ref{sec:theory}, we present our theoretical analysis of the proposed algorithm and compare it with the state-of-the-art Cubic Regularization methods. We conduct thorough numerical experiments on different nonconvex optimization problems and on different real world datasets to validate our theory in Section \ref{sec:exp}. We conclude our work in Section \ref{sec:conclusion}.

%% file: Algorithm.tex
\section{The Proposed Algorithm}\label{sec:alg}

\begin{algorithm*}[ht]
\caption{Sample efficient stochastic variance-reduced cubic regularization method (Lite-SVRC)}\label{algorithm:1}
\begin{algorithmic}[1]
  \STATE \textbf{Input:} batch size parameters $D_g,D_h$, penalty parameter $M_{s,t}$, $s\in \{1, ..., S\}, t\in\{0, ..., T-1\}$, initial point $\hat{\xb}^0$.
  \FOR{$s=1,\ldots,S$}
  \STATE $\xb_0^{s} = \hat{\xb}^{s-1}$
  \STATE $\gb^{s}=\nabla F(\hat{\xb}^{s-1})=\frac{1}{n}\sum_{i=1}^n \nabla f_i(\hat{\xb}^{s-1}),\Hb^{s} =\nabla^2 F(\hat{\xb}^{s-1})= \frac{1}{n}\sum_{i=1}^n \nabla^2 f_i(\hat{\xb}^{s-1})$
  \STATE $\hb_0^s = \argmin_{\hb\in \RR^d} m_0^s(\hb) =  \la\gb^s,\hb\ra + \frac{1}{2}\la\Hb^s \hb,\hb\ra + \frac{M_{s,0}}{6} \|\hb\|_2^3$
  \STATE $\xb_1^s = \xb_0^s + \hb_0^s$
  \FOR{$t=1,\ldots,T-1$}
  \STATE $b_{s,t} = D_g/\|\xold - \rx\|_2^2, t>0$ 
  \STATE $B_{s,t} = D_H$
  \STATE Sample index set $I_{g}, I_{h}\subseteq[n]$ uniformly with replacement, $|I_g| = b_{s,t}, |I_h| = B_{s,t}$
  \STATE $\vb_t^{s} = 1/b_{s,t}\big(\sum_{i_t \in I_g} \nabla f _{i_t}({\xb}_t^{s}) -  \nabla f _{i_t}(\hat{\xb}^{s-1})\big) +  {\gb}^{s} $
  \STATE $\Ub_t^{s} = 1/B_{s,t}\big(\sum_{j_t \in I_h} \nabla^2 f_{j_t}(\xb_t^{s}) - \nabla^2 f_{j_t}(\hat{\xb}^{s-1}) \big)+\Hb^s$
  \STATE ${\hb}_t^{s} = \argmin_{\hb\in \RR^d} m_t^s(\hb) =  \la{\vb}_t^{s},\hb\ra + \frac{1}{2}\la\Ub_t^{s} \hb,\hb\ra + \frac{M_{s,t}}{6} \|\hb\|_2^3$ 
  \STATE $\xb_{t+1}^{s} = \xb_{t}^{s} + \hb_t^{s}$
  \ENDFOR
  \STATE $\hat{\xb}^{s} = \xb_{T}^{s}$
  \ENDFOR
  \STATE \textbf{Output:} Uniformly randomly choose one $\xb_t^s$ as $\xout$, for $t \in\{1,\ldots,T\}$ and $s \in\{1,\ldots,S\}$.
\end{algorithmic}
\end{algorithm*}


In this section, we present our proposed algorithm Lite-SVRC. As is displayed in Algorithm \ref{algorithm:1}, 
our algorithm has $S$ epochs with each epoch length $T$. At the beginning of the $s$-th epoch, we calculate the gradient and Hessian of $F$ as `reference' of our algorithm, denoted by $\rg$ and $\rH$ respectively. Unlike $\CR$ which needs to calculate the full gradient and Hessian at each iteration, we only need to calculate them every $T$ iterations.

At the $t$-th iteration of the $s$-th epoch, we need to solve the CR sub-problem defined in \eqref{def:total_sub_problem}. Since the computational cost of $\dF(\xold)$ and $\HF(\xold)$ is expensive, we use the following semi-stochastic gradient $\vst$ and Hessian $\Ust$ instead
\begin{align}
    \vst 
     &= \frac{1}{\bst}\sum_{i_t \in I_g}\big[\df_{i_t}(\xold)  - \df_{i_t}(\rx) \big]+\rg, \label{def:vst}\\
    \Ust
    &= \frac{1}{\Bst}\sum_{j_t \in I_h}\big[\Hf_{j_t}(\xold)  -  \Hf_{j_t}(\rx) \big]+\rH, \label{def:Ust}
\end{align}
where $\rx$ is the reference point at which $\rg$ and $\rH$ are computed, $I_g$ and $I_h$ are sampling index sets (with replacement), $\bst$ and $\Bst$ are sizes of $I_g$ and $I_h$. Note that similar semi-stochastic gradient and Hessian have been proposed in \cite{johnson2013accelerating,xiao2014proximal} and  \cite{gower2017tracking,wai2017curvature,zhou2018stochastic,wang2018sample} respectively. 
We choose the minibatch sizes of stochastic gradient and stochastic Hessian for Algorithm \ref{algorithm:1} as follows:
\begin{align}
    \bst = D_g /\|\xold - \rx\|_2^2, \quad\Bst = D_H,
\end{align}
where $D_g, D_H$ are two constants only depending on $n$ and $d$. 

Compared with the SVRC algorithm proposed in \cite{zhou2018stochastic}, our algorithm uses a lite version of semi-stochastic gradient \citep{johnson2013accelerating,xiao2014proximal}, instead of the sophisticated one with Hessian information proposed in \cite{zhou2018stochastic}. Note that the additional Hessian information in the semi-stochastic gradient in \cite{zhou2018stochastic} actually increases the Hessian sample complexity. Therefore, with the goal of reducing the Hessian sample complexity, the standard semi-stochastic gradient \citep{johnson2013accelerating,xiao2014proximal} used in this paper is more favored.

On the other hand, there are two major differences between our algorithm and the SVRC algorithms proposed in \cite{wang2018sample}: (1) our algorithm uses a constant Hessian minibatch size instead of an adaptive one in each iteration, and thus the parameter tuning of our algorithm is much easier. In sharp contrast, the minibatch size of the stochastic Hessian in the algorithm proposed by \cite{wang2018sample} is dependent on the next iterate, which makes the update an implicit one and it is hard to tune the parameters in practice; and (2) our algorithm does not need to compute the minimum eigenvalue of the Hessian in each iteration, and thus really reduces the Hessian sample complexity as well as runtime complexity in practice. In contrast, the algorithm in \cite{wang2018sample} needs to calculate the minimum eigenvalue of the Hessian as a stopping criteria in each iteration, which actually incurs additional $O(n)$ Hessian sample complexity. 


%% file: theory.tex
\section{Main Theory}\label{sec:theory}
In this section, we present our theoretical results on the Hessian sample complexity of Lite-SVRC.

We start with the following assumptions that are needed throughout our analysis:
\begin{assumption}[Gradient Lipschitz]\label{assumption:grad_lip}
 There exists a constant $\glip >0$, such that for all $\xb,\yb$ and $i \in \{1,...,n\}$ 
\begin{align*}
  \|\df_i(\xb) - \df_i (\yb)\|_2 \leq \glip\|\xb-\yb\|_2  .
\end{align*}
\end{assumption}

\begin{assumption}[Hessian Lipschitz]\label{assumption:hess_lip}
 There exists a constant $\Hlip >0$, such that for all $\xb,\yb$ and $i \in \{1,...,n\}$ 
\begin{align*}
  \|\Hf_i(\xb) - \Hf_i (\yb)\|_2 \leq \Hlip\|\xb-\yb\|_2  .
\end{align*}
\end{assumption}

These two assumptions are mild and widely used in the line of research for finding approximate global minima \citep{Carmon2016Gradient, Carmon2016Accelerated,Agarwal2017Finding,wang2018sample}.
Next we present two key definitions, which play important roles in our analysis:
\begin{definition}
We define the optimal gap as
\begin{align}
    \Delta_F = F(\hat{\xb}^0) - \inf_{\xb \in \RR^d} F(\xb).
\end{align}
\end{definition}
\begin{definition}\label{def:mu}
Let $\xold$ be the iterate defined in Algorithm \ref{algorithm:1}, where $s\in \{1, ..., S\}$ and $t\in\{0, ..., T-1\}$. We define $\mu(\xnew)$ as follows:
\begin{align}
\mu(\xnew) &= \max \big\{\|\dF(\xnew)\|_2^{3/2}, -\big(\lambda_{\min} \big(\HF(\xnew)\big)\big)^3\Mst^{-3/2}\big\}.
\end{align}
\end{definition}

Definition \ref{def:mu} appears in \cite{Nesterov2006Cubic} with a slightly different form, which is used to describe how much a point $\xnew$ is similar to a true local minimum. 
Recall the definition of approximate local minima in \eqref{eq:def_localminima}, it is easy to show the following fact: if $\Mst = \bigO(\Hlip)$ holds for any $s,t$, then $\xnew$ is an $(\epsilon, \sqrt{\Hlip\epsilon})$-approximate local minimum if and only if $\mu(\xnew) = \bigO(\epsilon^{3/2})$. We note that similar argument is also made in \cite{zhou2018stochastic}.

From now on, we will focus on bounding $\mu(\xnew)$, which is equivalent to finding the approximate local minimum. The following theorem spells out the upper bound of $\mu(\xnew)$.
\begin{theorem}\label{thm:1}
Under Assumptions \ref{assumption:grad_lip} and \ref{assumption:hess_lip}, suppose that $n>10, \Mst>2\Hlip$ and $D_h>25\log d$. Let $\alphast, \betast>0$ be arbitrary chosen parameters, and $\Gst$ and $ \cold$ are positive parameters satisfying following induction equations for all $s\in \{1, ..., S\}$ and  $t\in\{0, ..., T-1\}$:
\begin{align}
    \Gst &=\frac{\Mst - 12\cnew(1+2\alphast+\betast)}{12\Mst^{3/2}}, \label{def_gst}\\
    \cold &= \frac{(\constant_1+\Gst\Mst^{1/2})\Hlip^{3/2}}{\Mst^{1/2}}\bigg(\bigg[\frac{4L_1^2}{D_gL_2^2}\bigg]^{3/4}+\frac{\constant_h\Hlip^{3/2}(\log d)^{3/2}}{\Mst^{3/2}D_h^{3/2}}\bigg)+\cnew\bigg(1+\frac{1}{\alphast^2}+\frac{2}{\betast^{1/2}}\bigg),\label{def_cst}\\
   c_{s,T} &= 0,\notag
\end{align}
where $\constant_h, \constant_1$ are absolute constants. Then the output of Algorithm \ref{algorithm:1} satisfies the following inequality
\begin{align}
     \EE \mu(\xout) \leq \frac{\Delta_F}{ST\minG},\label{bound_mu}
\end{align}
where $\minG$ is defined as follows
\begin{align*}
    \minG = \min_{s\in \{1, ..., S\}, t\in\{0, ..., T-1\}} \frac{\Gst}{\constant_\mu}, 
\end{align*}
and $\constant_\mu$ is an absolute constant.
\end{theorem}



\begin{remark}
Theorem \ref{thm:1} suggests that with a fixed number of inner loops $T$, if we run Algorithm \ref{algorithm:1} for sufficiently large $S$ epochs, then we have a point sequence $\xb_i$ where $\EE\mu(\xb_i)\rightarrow 0$. That being said, $\xb_i$ will converge to a local minimum, which is consistent with the convergence analysis in existing related work \citep{Nesterov2006Cubic, kohler2017sub, wang2018sample}. 
\end{remark}

Now we give a specific choice of parameters mentioned in Theorem \ref{thm:1} to derive the total Hessian sample complexity of Algorithm \ref{algorithm:1}.

\begin{corollary}\label{coro:result_come}
Under the same conditions as in Theorem \ref{thm:1}, let batch size parameters satisfy $D_g = 4L_1^2/L_2^2\cdot n^{4/3}$ and $D_h = \log d\cdot(\constant_h\cdot n)^{2/3}$. Set the inner loop parameter $T = n^{1/3}$ and cubic penalty parameter $\Mst = \constant_m \Hlip$, where $\constant_m$ is an absolute constant. Then the output $\xout$ from Algorithm \ref{algorithm:1} is a $(\epsilon, \sqrt{\Hlip\epsilon})$-approximate local minimum after 
\begin{align}
    \tbigO\bigg(n+\frac{\Delta_F\sqrt{\Hlip}}{\epsilon^{3/2}}\cdot n^{2/3}\bigg)
\end{align}
stochastic Hessian evaluations. 
\end{corollary}

Now we provide a comprehensive comparison between our algorithm and other related algorithms in Table \ref{table:perbatch}. The algorithm proposed in \cite{wang2018sample} has two versions: sample with replacement and sample without replace. For the completeness, we present both versions in \cite{wang2018sample}.
From Table \ref{table:perbatch} we can see that Lite-SVRC strictly outperforms CR by a factor of $n^{1/3}$. 
Lite-SVRC also outperforms SCR when $\epsilon = \bigO(n^{-2/3})$, which suggests that the variance reduction scheme makes Lite-SVRC perform better in the high accuracy regime. More importantly, our proposed Lite-SVRC does not rely on the assumption that the function $F$ is Lipschitz continuous, which is required by the algorithm proposed in \cite{wang2018sample}. So in terms of Hessian sample complexity, our algorithm directly improves that of  \cite{wang2018sample} by a factor of $n^{2/33}$.

%% file: Experiments.tex
\section{Experiments}\label{sec:exp}



In this section, we conduct experiments on real world datasets to support our theoretical analysis of the proposed $\SVRCp$ algorithm. Following \cite{zhou2018stochastic}, we investigate two nonconvex problems on three different datasets, \emph{a9a} (sample size: $32,561$, dimension: $123$), \emph{ijcnn1} (sample size: $35,000$, dimension: $22$) and \emph{covtype} (sample size: $581,012$, dimension: $54$), which are all common datasets used in machine learning. 


\begin{table}[ht]
\caption{Comparisons of per-iteration and total sample complexities of Hessian evaluations for different algorithms.\label{table:perbatch}}
\begin{small}
\begin{center}
\begin{tabular}{cccccc}
\toprule
\multirow{ 2}{*}{\textbf{algorithm}} & \multirow{ 2}{*}{\textbf{per-iteration}} & \multirow{ 2}{*}{\textbf{total}} & \textbf{function} & \textbf{gradient} & \textbf{Hessian} \\
&&&\textbf{Lipschitz}&\textbf{Lipschitz}&\textbf{Lipschitz}\\
\midrule
CR & \multirow{2}{*}{$\bigO(n)$}  & \multirow{2}{*}{$\bigO\Big(\frac{n}{\epsilon^{3/2}}\Big)$} &  \multirow{2}{*}{No} & \multirow{2}{*}{No} &  \multirow{2}{*}{Yes}\\
\citet{Nesterov2006Cubic} & & & &\\
SCR & \multirow{3}{*}{$\tbigO\Big(\frac{1}{\epsilon}\Big)$}  & \multirow{3}{*}{$\tbigO\Big(\frac{1}{\epsilon^{5/2}}\Big)$} &  \multirow{3}{*}{No\footnote[3]} & \multirow{3}{*}{Yes} &  \multirow{3}{*}{Yes}\\
\citet{kohler2017sub} & & & &\\
\citet{xu2017second} & & & &\\
SVRC  & \multirow{2}{*}{$\tbigO(n^{4/5})$} & \multirow{2}{*}{$\tbigO\Big(n+\frac{n^{4/5}}{\epsilon^{3/2}}\Big)$} &  \multirow{2}{*}{No} & \multirow{2}{*}{No} &  \multirow{2}{*}{Yes}\\
\cite{zhou2018stochastic}\footnote[4]  & & & &\\ 
SVRC$_{\text{with}}$  & \multirow{2}{*}{$\tbigO\Big(\frac{\|\xold - \rx\|_2^2}{\|\hst\|_2^2}\Big)$} & \multirow{2}{*}{$\tbigO\Big(n+\frac{n^{3/4}}{\epsilon^{3/2}}\Big)$} &  \multirow{2}{*}{Yes} & \multirow{2}{*}{Yes} &  \multirow{2}{*}{Yes}\\
\cite{wang2018sample}\footnote[5] & &&&\\
SVRC$_{\text{without}}$   
& \multirow{2}{*}{$\tbigO\Big(\Big(\frac{1}{n}+\frac{\|\hst\|_2^2}{\|\xold -\rx\|_2^2 }\Big)^{-1}\Big)$} & \multirow{2}{*}{$\tbigO\Big(n+\frac{n^{8/11}}{\epsilon^{3/2}}\Big)$}&  \multirow{2}{*}{Yes} & \multirow{2}{*}{Yes} &  \multirow{2}{*}{Yes} \\
\cite{wang2018sample}\footnote[5]  & &&&\\
$\SVRCp$   & \multirow{2}{*}{$\tbigO(n^{2/3})$}& \multirow{2}{*}{$\tbigO\Big(n+\frac{n^{2/3}}{\epsilon^{3/2}}\Big)$} &  \multirow{2}{*}{No} & \multirow{2}{*}{Yes} &  \multirow{2}{*}{Yes}\\
(This paper)  & & & &\\
\bottomrule
\end{tabular}
\end{center}
\end{small}
\end{table}

\footnotetext[3]{Although the refined SCR in \cite{xu2017newton} does not need function Lipschitz, the original SCR in \cite{kohler2017sub} needs it.}
\footnotetext[4]{We adapt this result directly from the analysis of total second-order oracle calls from \cite{zhou2018stochastic}.} 
\footnotetext[5]{In \cite{wang2018sample}, both algorithms need to calculate $\mineig(\HF(\xold))$ at each iteration to decide whether the algorithm should continue, which adds additional $\bigO(n)$ Hessian sample complexity. We choose not to include this into the results in the table.}

\subsection{Baseline Algorithms}
To evaluate our proposed algorithm, we compare the proposed $\SVRCp$ with the following $7$ baseline algorithms: (1) trust-region Newton methods (denoted by TR) \cite{conn2000trust}; (2) Adaptive Cubic regularization \citep{Cartis2011Adaptive,Cartis2011Adaptive2}; (3) Subsampled Cubic regularization \citep{kohler2017sub}; (4) Gradient Cubic regularization \citep{Carmon2016Gradient}; (5) Stochastic Cubic regularization \citep{tripuraneni2017stochastic}; (6) SVRC proposed in \cite{zhou2018stochastic}; (7) SVRC-without proposed in \cite{wang2018sample}. Note that there are two versions of SVRC algorithm proposed in \cite{wang2018sample}, and the one based on sampling without replacement performs better in both theory and experiments, we therefore only compare with this one, which is denoted by SVRC-without. 

\subsection{Implementation Details}
For Subsampled Cubic and SVRC-without, the sample size $B_k$ is dependent on $\|\hb_k\|_2$ \citep{kohler2017sub} and $\Bst$ is dependent on $\|\hst\|_2$ \citep{wang2018sample}, which make these two algorithms implicit algorithms. To address this issue, we follow the suggestion in \cite{kohler2017sub, wang2018sample} and use $\|\hb_{k-1}\|_2$ and $\|\hb_{s,t-1}\|_2$ instead of $\|\hb_k\|_2$ and $\|\hst\|_2$. Furthermore, we choose the penalty parameter $\Mst$ for SVRC, SVRC-without and Lite-SVRC as constants which are suggested by the original papers of these algorithms. Finally, to solve the CR sub-problem in each iteration, we choose to solve the sub-problem approximately in the Krylov sub-space spanned by Hessian related vectors, as used by \cite{kohler2017sub}. 


In the experiment, we choose two nonconvex regression problem as our objectives. Both of them consist of a loss function (can be nonconvex) and the following nonconvex regularizer
\begin{align}\label{eq:noncon_penalty}
    g(\lambda, \alpha, \xb) = \lambda\cdot \sum_{i=1}^d \frac{(\alpha x_i)^2}{1+(\alpha x_i)^2},
\end{align}
where $\lambda, \alpha$ are the control parameters and $x_i$ is the $i$-th coordinate of $\xb$. This regularizer has been widely used in nonconvex regression problem, which can be regarded as a special example of robust nonlinear regression \citep{reddi2016fast,kohler2017sub,zhou2018stochastic, wang2018sample}.

\begin{figure*}[ht]
\vskip -0.2in
	\begin{center}
		\subfigure[\textit{a9a}]{\includegraphics[width=0.32\linewidth]{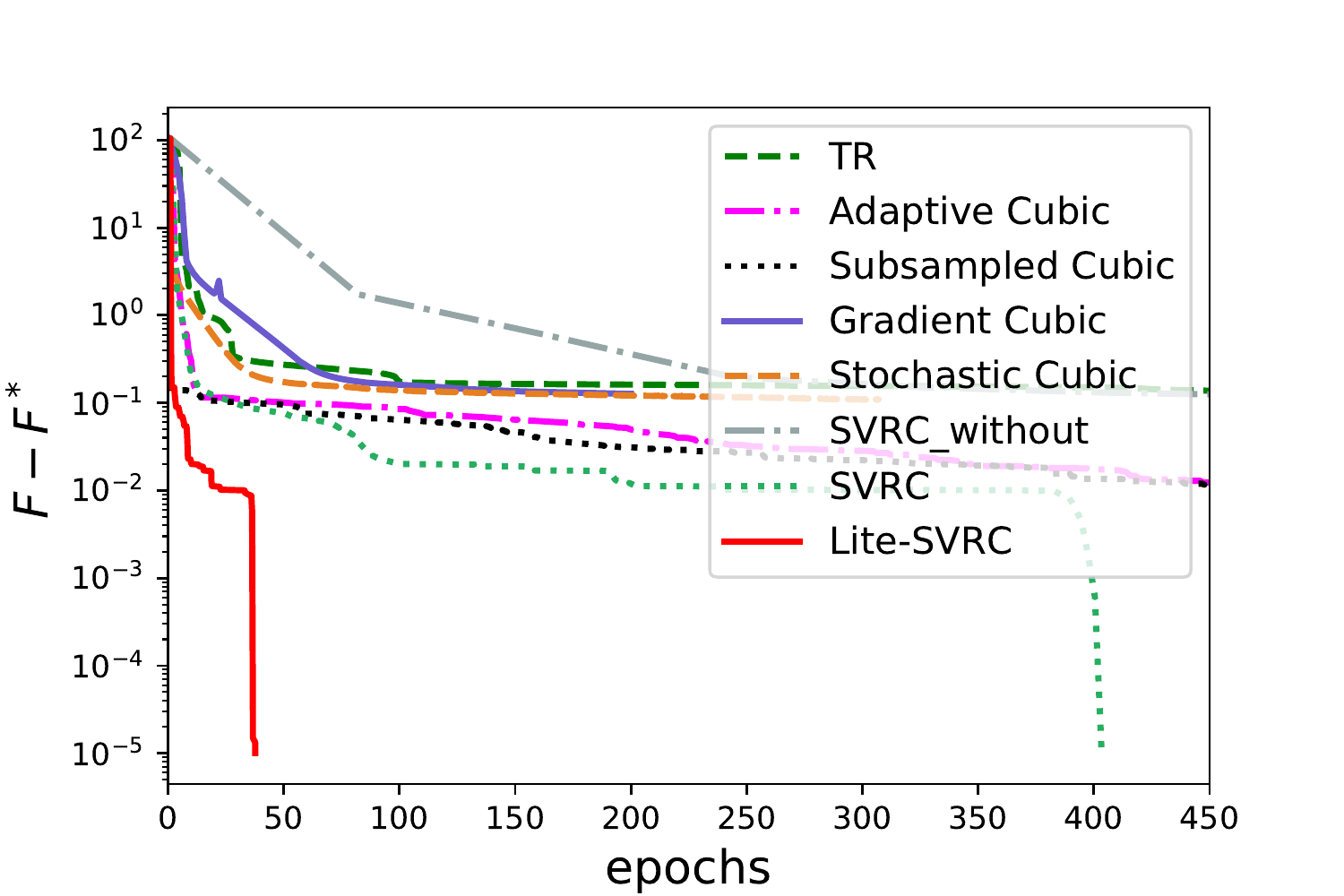}}	
		\subfigure[\textit{ijcnn1}]{\includegraphics[width=0.32\linewidth]{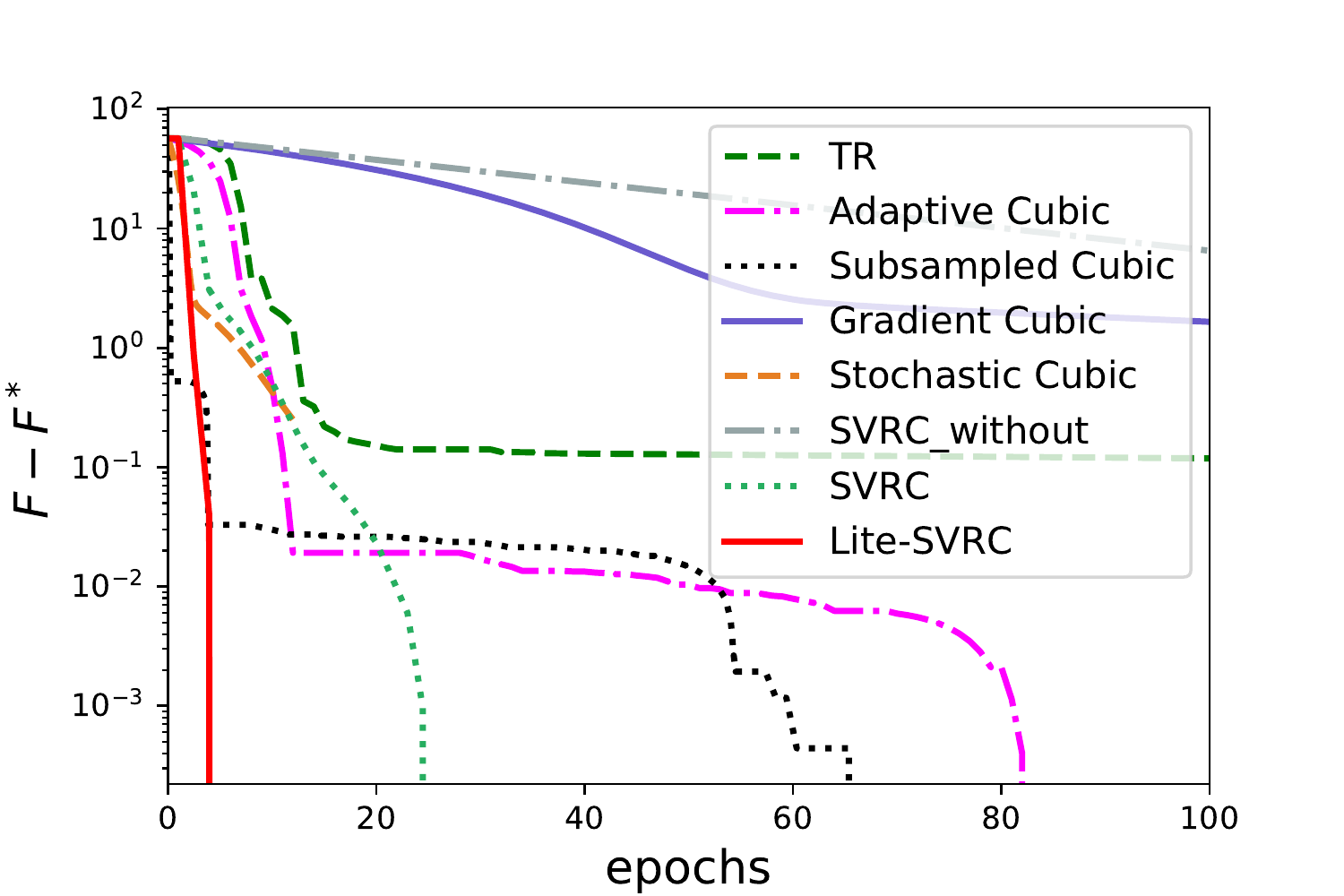}\label{fig:simu_logis_entropy}}
		\subfigure[\textit{covtype}]{\includegraphics[width=0.32\linewidth]{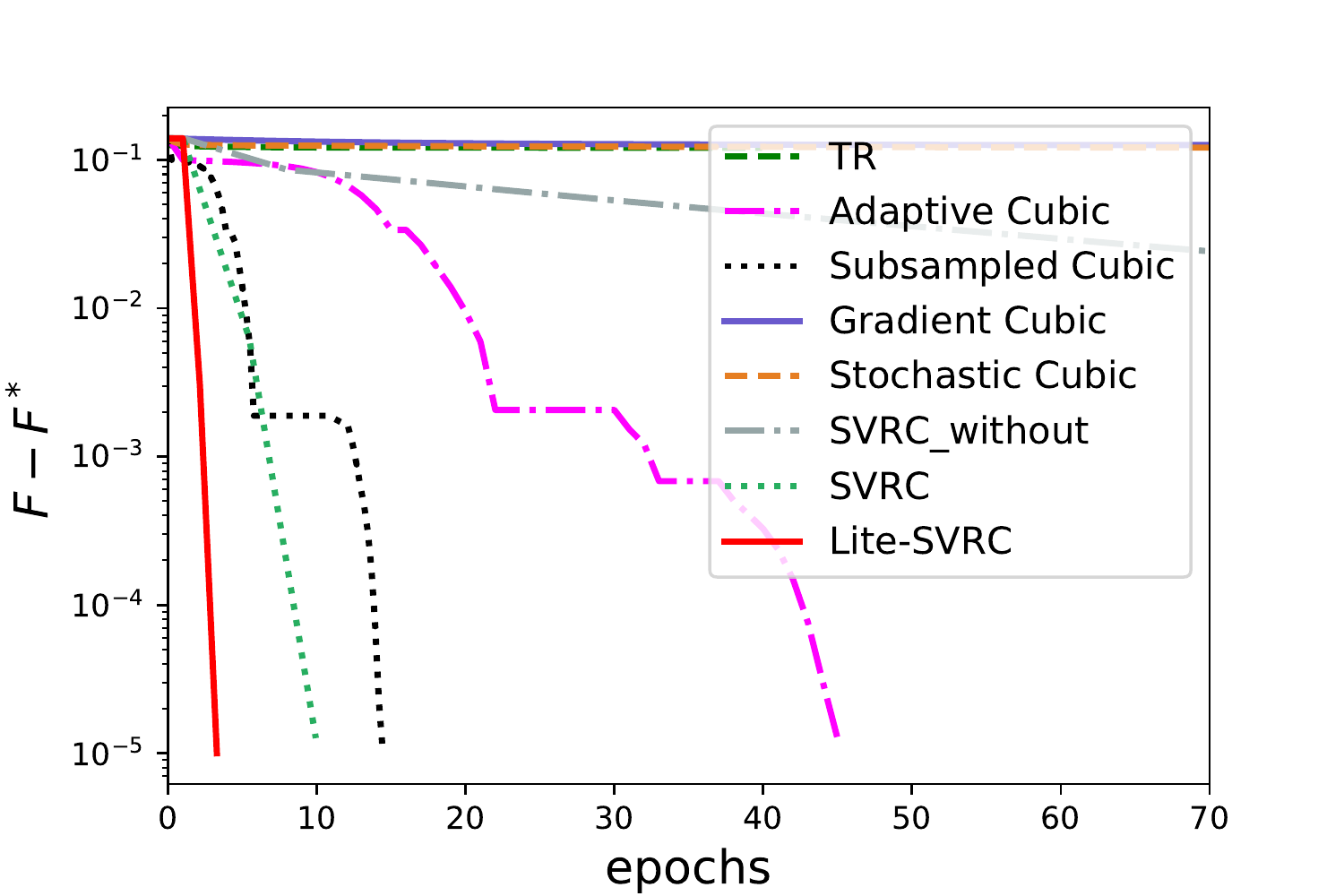}}
		
		\subfigure[\textit{a9a}]{\includegraphics[width=0.32\linewidth]{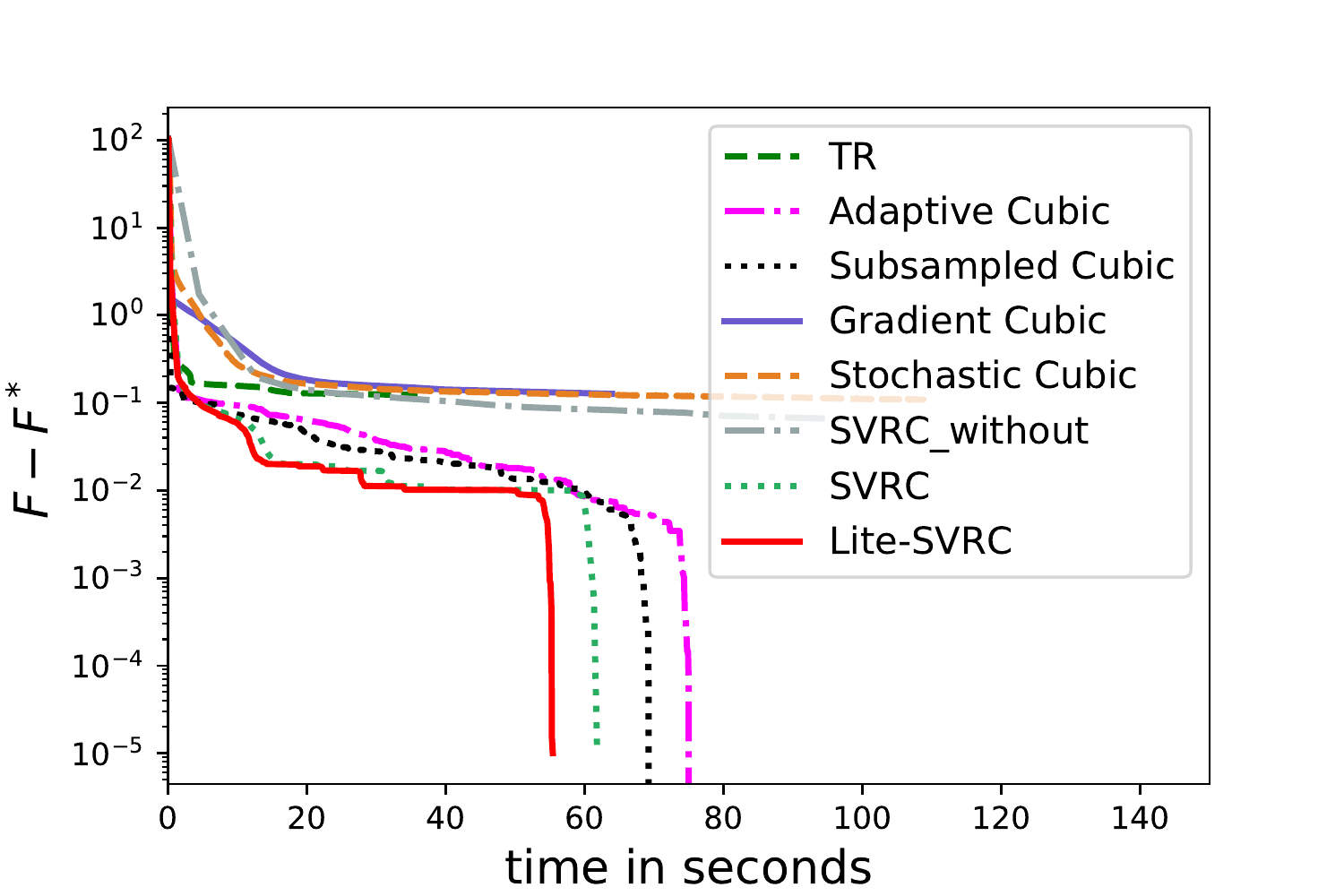}}
			\subfigure[\textit{ijcnn1}]{\includegraphics[width=0.32\linewidth]{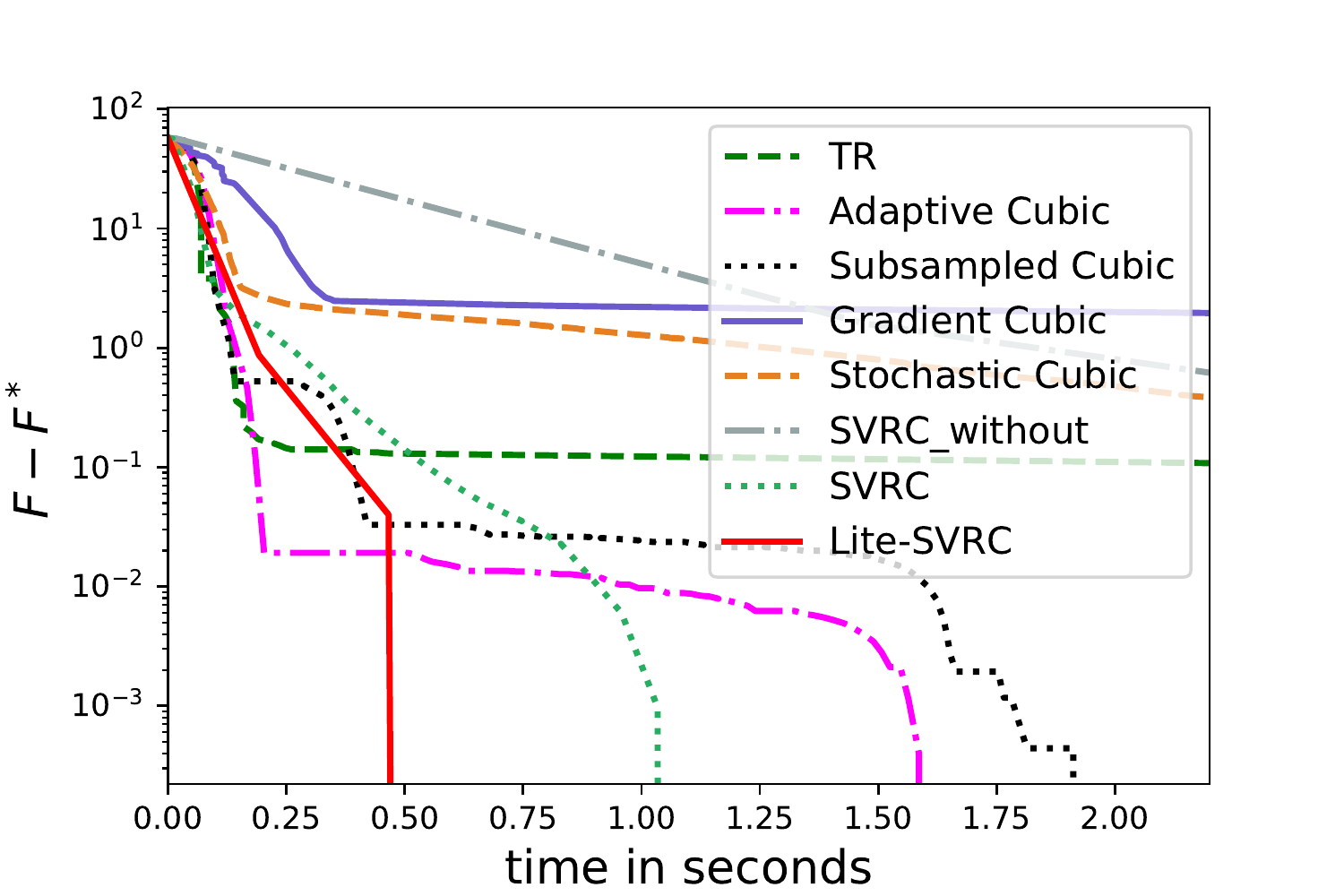}\label{fig:simu_logis_entropy}}
    	\subfigure[\textit{covtype}]{\includegraphics[width=0.32\linewidth]{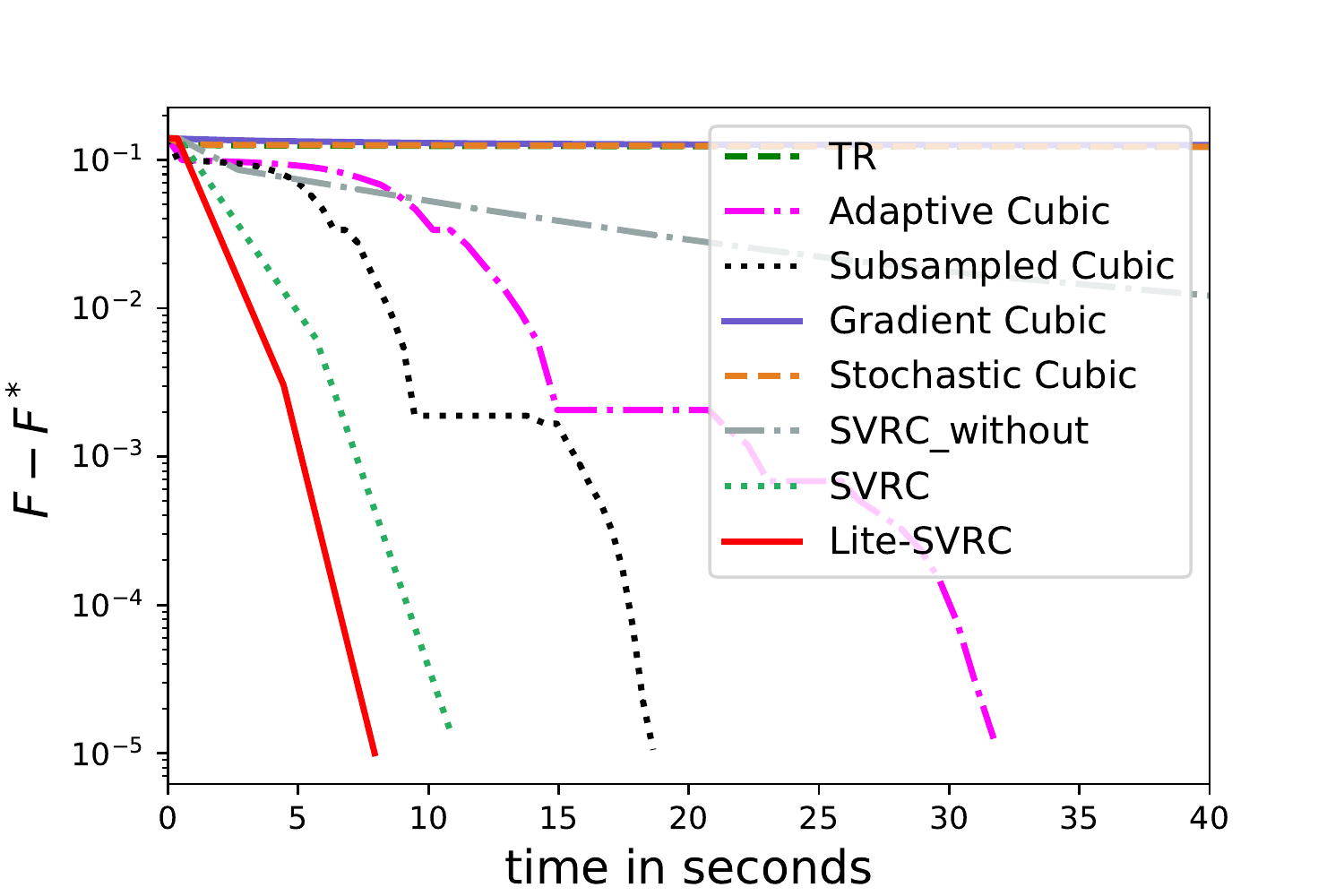}}
		
	\caption{Function value gap of different algorithms for nonconvex regularized logistic regression problems on different datasets. (a)-(c) are plotted w.r.t. Hessian sample complexity. (d)-(e) are plotted w.r.t. CPU runtime.\label{fig:reg_logistic}}
	\end{center}
\end{figure*}
\subsection{Logistic Regression with Nonconvex Regularizer}
The first problem is a binary logistic regression problem with a nonconvex regularizer $g(\lambda, \alpha \xb)$. Given training data $\xb_i \in \RR^d$ and label $y_i \in \{0,1\}$, $1 \leq i \leq n$, our goal is to solve the following optimization problem:
\begin{align}
    \min_{\sbb \in \RR^d}& \frac{1}{n}\sum_{i=1}^n\big[y_i\cdot  \log \phi(\sbb^{\top}\xb_i) + (1-y_i)\cdot \log [1- \phi(\sbb^{\top} \xb_i)]\big] + g(\lambda, \alpha, \sbb), 
\end{align}
where $\phi(x) = 1/(1+\exp(-x))$ is the sigmoid function and $\lambda$ and $\alpha$ are the parameters that are used to define the nonconvex regularizers in \eqref{eq:noncon_penalty} and are set differently for each dataset. In detail, we set $\lambda=10^{-3}$ for all three datasets, and set $\alpha=10,50,100$ for \emph{a9a}, \emph{ijcnn1} and \emph{covtype} datasets respectively. 

The experiment results on the binary logistic regression problem are displayed in Figure \ref{fig:reg_logistic}. The first row of the figure shows the plots of function value gap v.s. Hessian sample complexity of all the compared algorithms, and the second row presents the plots of function value gap v.s. CPU runtime (in second) of all the algorithms. It can be seen from Figure \ref{fig:reg_logistic} that Lite-SVRC performs the best among all algorithms regarding both sample complexity of Hessian and runtime on all three datasets, which is consistent with our theoretical analysis. We remark that SVRC performs the second best in most settings in terms of both Hessian sample complexity and runtime. It should also be noted that although SVRC-without is also a variance-reduced method similar to Lite-SVRC and SVRC, it indeed performs much worse than other methods, because as we pointed out in the introduction, it needs to compute the minimum eigenvalue of the Hessian in each iteration, which actually makes the Hessian sample complexity even worse than Subsampled Cubic, let alone the runtime complexity. 

\begin{figure*}[hbt]
	\begin{center}
		\subfigure[\textit{a9a}]{\includegraphics[width=0.32\linewidth]{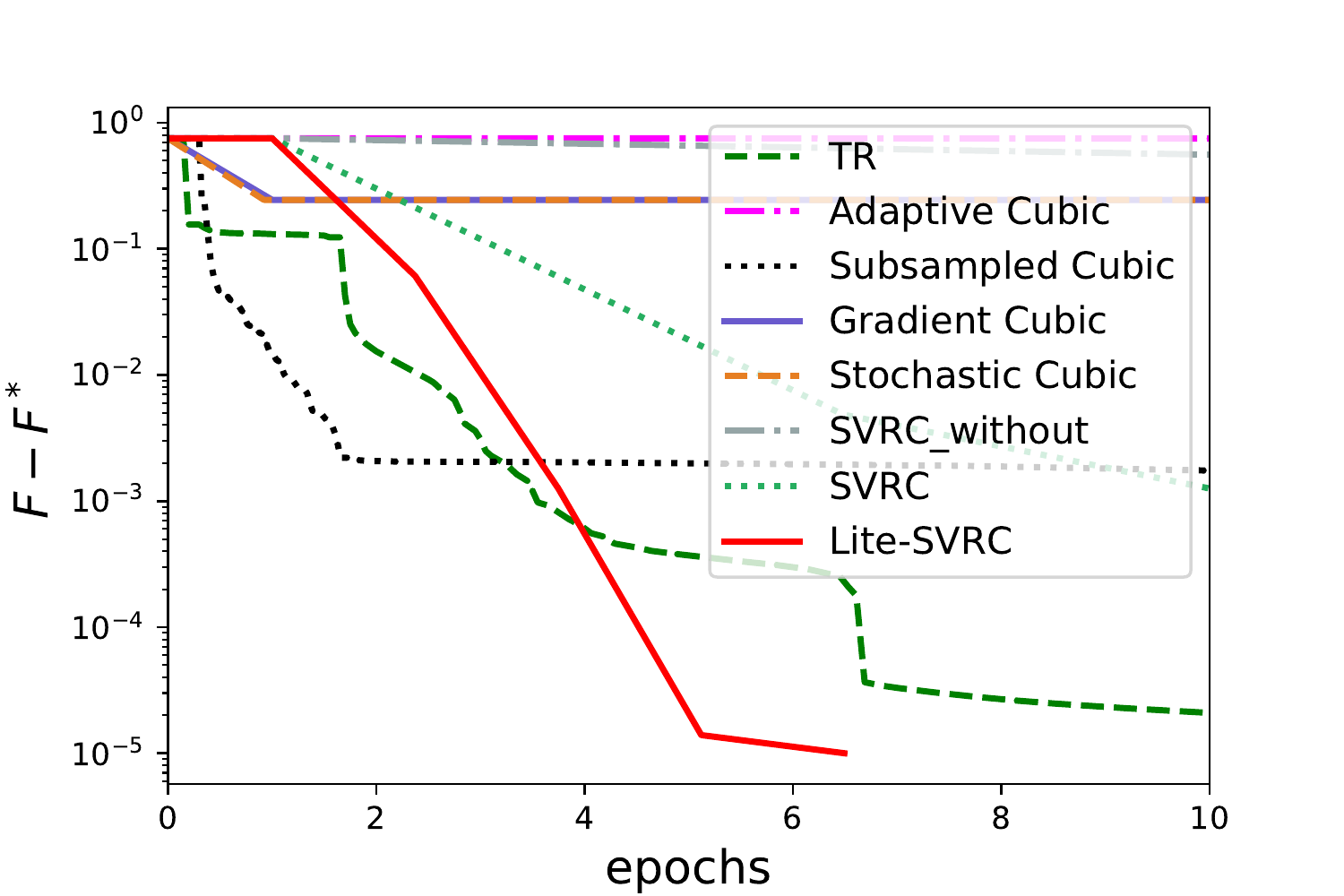}}	
		\subfigure[\textit{ijcnn1}]{\includegraphics[width=0.32\linewidth]{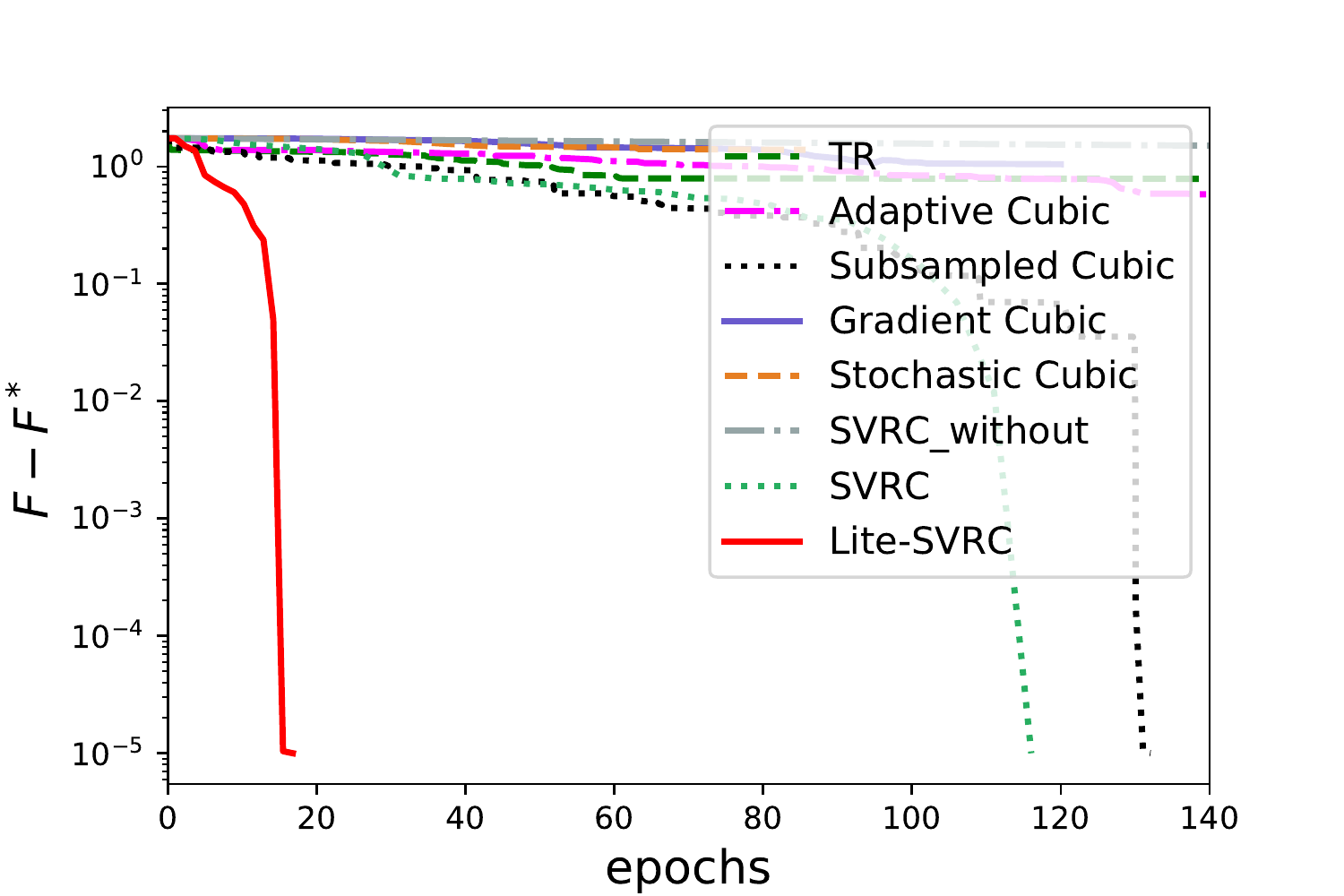}\label{fig:simu_logis_entropy}}
		\subfigure[\textit{covtype}]{\includegraphics[width=0.32\linewidth]{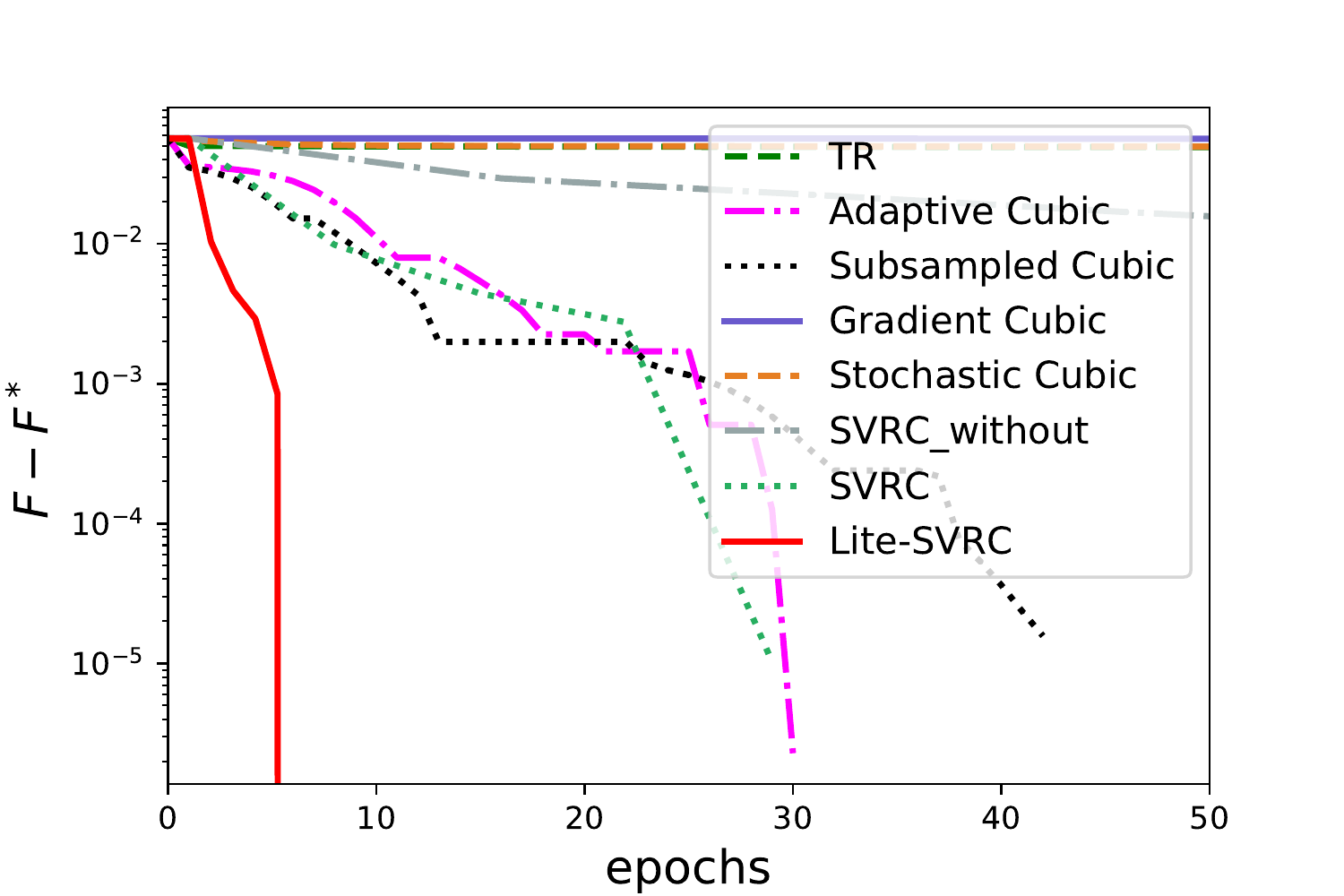}}
		
		\subfigure[\textit{a9a}]{\includegraphics[width=0.32\linewidth]{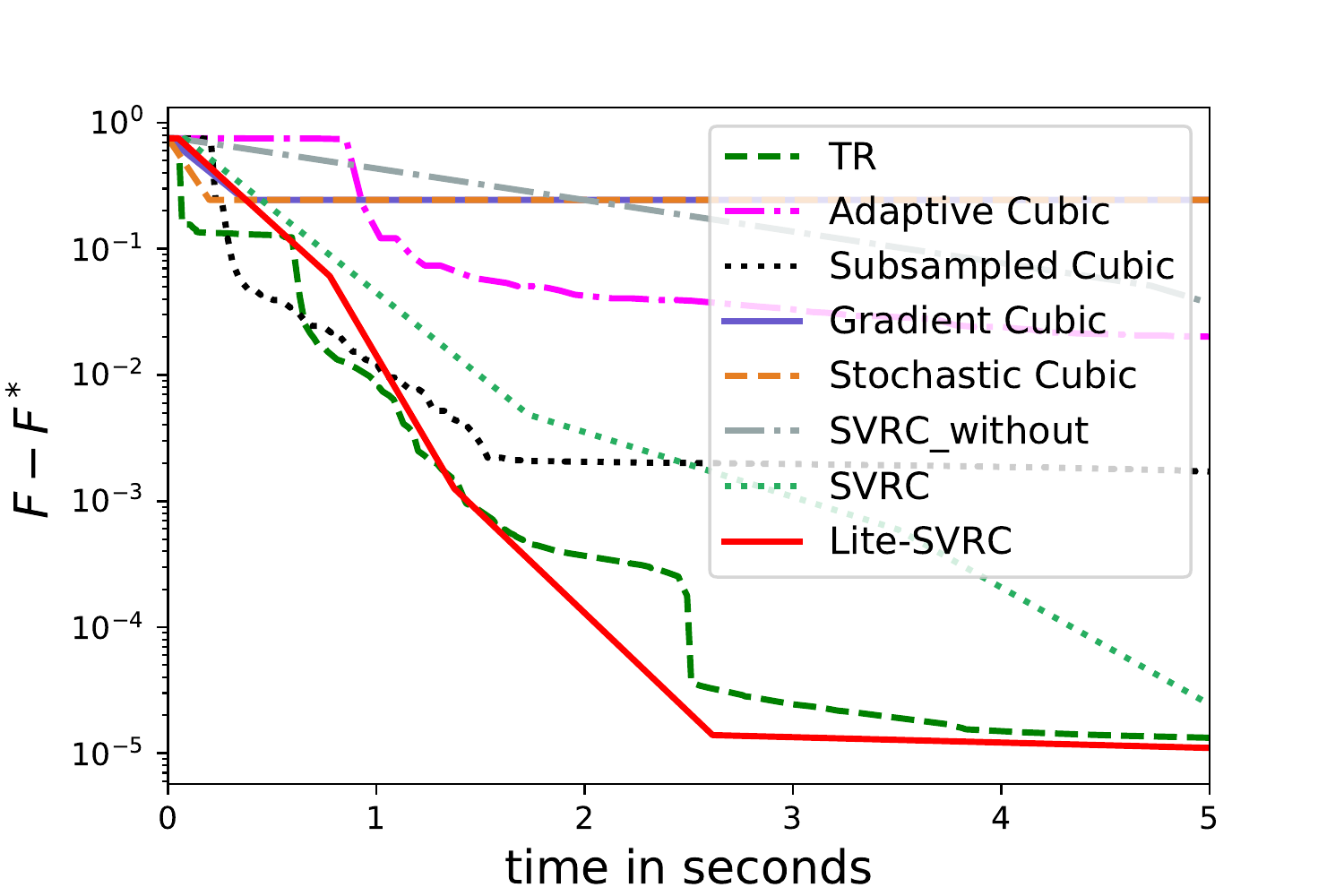}}	
		\subfigure[\textit{ijcnn1}]{\includegraphics[width=0.32\linewidth]{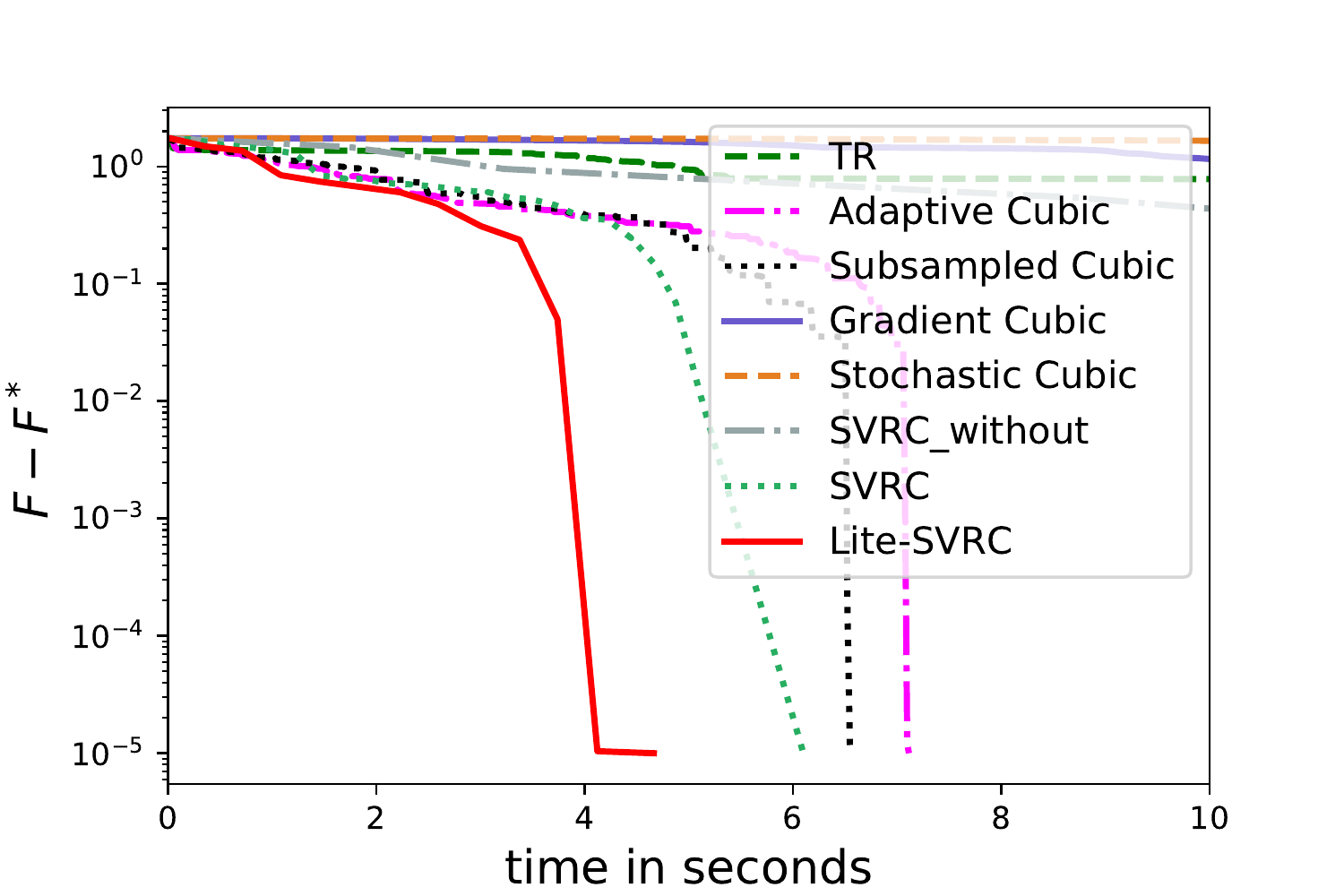}\label{fig:simu_logis_entropy}}	
		\subfigure[\textit{covtype}]{\includegraphics[width=0.32\linewidth]{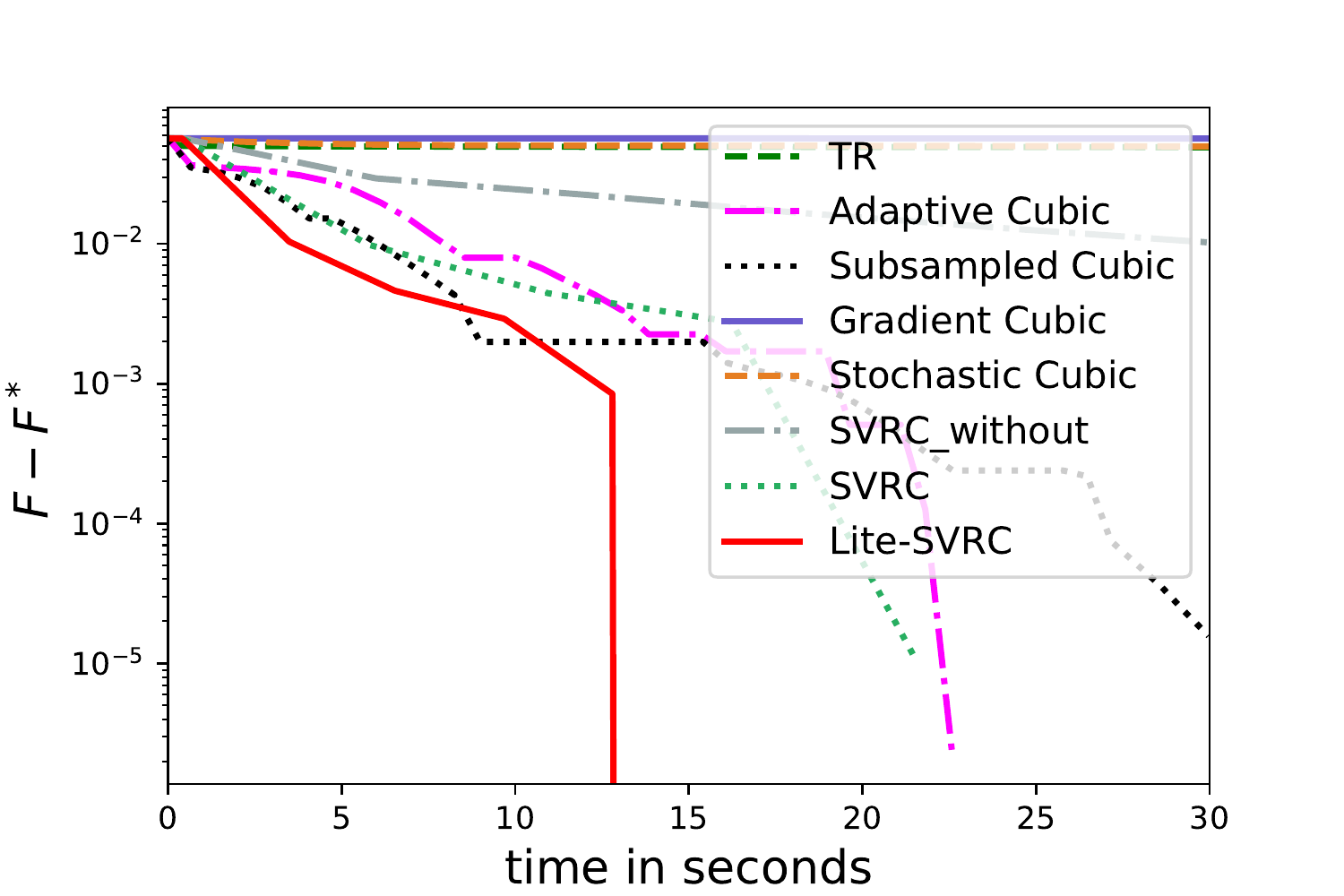}}
	\caption{Function value gap of different algorithms for nonlinear least square problems on different datasets. (a)-(c) are plotted w.r.t. Hessian sample complexity. (d)-(e) are plotted w.r.t. CPU runtime. \label{fig:nonlinear_ls}}
	\end{center}
	\vskip -0.2in
\end{figure*}

\subsection{Nonlinear Least Square with Nonconvex Regularizer}

In this subsection, we consider another problem, namely, the nonlinear least square problem with a nonconvex regularizer $g(\lambda, \alpha, \xb)$ defined in \eqref{eq:noncon_penalty}. The nonlinear least square problem is also studied in \cite{xu2017newton, zhou2018stochastic}. Given training data $\xb_i \in \RR_d$ and $y_i \in \{0,1\}$, $1\leq i \leq n$, our goal is to minimize the following problem
\begin{align}
    \min_{\sbb \in \RR^d}\frac{1}{n}\sum_{i=1}^n\big[y_i - \phi(\sbb^{\top}\xb_i)\big]^2+g(\lambda, \alpha, \sbb).
\end{align}
Here $\phi(x) = 1/(1+\exp(-x))$ is again the sigmoid function. The parameters $\lambda$ and $\alpha)$ in the nonconvex regularizer for different datasets are set as follows: we set $\lambda=5\times10^{-3}$ for all three datasets, and set $\alpha=10,20,50$ for \emph{a9a}, \emph{ijcnn1} and \emph{covtype} datasets respectively. The experiment results are summarized in Figure \ref{fig:nonlinear_ls}, where the first row shows the plots of function value gap v.s. Hessian sample complexity and the second row presents the plots of function value v.s. CPU runtime (in second). It can be seen that Lite-SVRC again achieves the best performance among all other algorithms regarding to both sample complexity of Hessian and runtime when the required precision is high, which supports our theoretical analysis again. SVRC performs the second best. 

%% file: Conclusion.tex
\section{Conclusions}\label{sec:conclusion}
In this paper, we propose a new algorithm called Lite-SVRC, which achieves lower sample complexity of Hessian compared with existing variance reduction based cubic regularization algorithms \citep{zhou2018stochastic,wang2018sample}. 
Extensive experiments on various nonconvex optimization problems and datasets validate our theory.

%% file: appendix.tex
\section{Proof of the Main Theory}\label{sec:proof_sketch}
In this section, we provide the proofs of Theorem \ref{thm:1} and Corollary \ref{coro:result_come}. 
\subsection{Proof of Theorem \ref{thm:1}}
Since our algorithm consists of inner loops and outer loops, we mainly focus on the analysis of one single step which is the $t$-th iteration in the $s$-th epoch, 
where $s\in \{1, ..., S\}, t\in\{0, ..., T-1\}$. 

Similar to other CR related work \citep{Nesterov2006Cubic, Cartis2011Adaptive, kohler2017sub}, our ultimate goal is to prove the following statement of one single loop:
\begin{align}
    \mu(\xnew) \leq \sqrt{\Hlip}\big[F(\xold) - F(\xnew)\big].\label{sketch:fake}
\end{align}
If \eqref{sketch:fake} holds, then we just take summation over $t$ and $s$ in the above inequality, which yields the final result of Theorem \ref{thm:1}.
Unfortunately, \eqref{sketch:fake} does not hold in general because of the existence of randomness in our algorithm. Nevertheless, by borrowing the idea from the analysis of SVRG in nonconvex setting \citep{Reddi2016Stochastic}, we propose to replace the function $F$ in \eqref{sketch:fake} with the following Lyapunov function:
\begin{align}
    F(\xold)+\cold\|\xold - \rx\|_2^3, \label{def:lya}
\end{align}
where $\cold$ are parameters defined in Theorem \ref{thm:1}. With the Lyapunov function in \eqref{def:lya}, we are able to prove the following key lemma that resembles \eqref{sketch:fake} and holds in expectation:
\begin{lemma}\label{lemma_oneloop}
Under the same assumption as in Theorem \ref{thm:1}, let $\xnew, \xold$ be variables defined in Algorithm \ref{algorithm:1}. $\Gst, \cold$ are parameters defined in Theorem \ref{thm:1}, and $\constant_\mu$ is a constant. Then we have following result:
\begin{align}\label{lemma_oneloop_eq}
    \frac{\Gst}{\constant_\mu}\cdot \EE(\mu(\xnew)) 
    & \leq\EE\big[F(\xold)+\cold\cdot \|\xold - \rx\|_2^3\big] - \EE\big[F(\xnew)+\cnew\cdot \|\xnew - \rx\|_2^3\big], 
\end{align}
where $\EE$ takes over all randomness.
\end{lemma}
With Lemma \ref{lemma_oneloop}, we are ready to deliver the proof of our main theory. 
\begin{proof}[Proof of Theorem \ref{thm:1}]
Applying Lemma \ref{lemma_oneloop}, we sum up \eqref{lemma_oneloop_eq} from $t = 0$ to $T-1$, while yields
\begin{align}\label{main0_0}
    \sum_{t=0}^{T-1} \frac{\Gst}{\constant_\mu}\cdot \EE(\mu(\xnew)) 
    &\leq \EE\big[F(\xb_0^s)+c_{s,0}\cdot \|\xb_0^s - \rx\|_2^3\big] - \EE\big[F(\xb_T^s)+c_{s,T}\cdot \|\xb_T^s - \rx\|_2^3\big].
\end{align}
Substituting $\xb_0^s = \rx, \xb_T^s = \hat{\xb}^s$ and $c_{s,T} = 0$ into \eqref{main0_0}, we get
\begin{align*}
    \sum_{t=0}^{T-1} \Gst/\constant_\mu\cdot \EE(\mu(\xnew)) \leq \EE F(\rx) - \EE F(\hat{\xb}^s).
\end{align*}
Then we take summation from $s = 1$ to $S$, we have
\begin{align*}
    \sum_{s=1}^S\sum_{t=0}^{T-1} \frac{\Gst}{\constant_\mu}\cdot \EE(\mu(\xnew)) &\leq \EE F(\hat{\xb}^0) - \EE F(\hat{\xb}^S)  \\
    &\leq F(\hat{\xb}^0) - \inf_{\xb\in \RR^d} F(\xb) \\
    &= \Delta_F.
\end{align*}
Because $\minG \leq \Gst/\constant_\mu$ for all $s=1...S, t=0...T-1$, we have
\begin{align}\label{main0_2}
    \minG\cdot\sum_{s=1}^S\sum_{t=0}^{T-1} \EE(\mu(\xnew)) \leq \Delta_F.
\end{align}
Finally, because we choose $\xout$ randomly over $s$ and $t$, thus we have our result from \eqref{main0_2}:
\begin{align*}
    \EE \mu(\xout) \leq \Delta_F /(ST\minG). 
\end{align*}
This competes the proof.
\end{proof}

\subsection{Proof of Corollary \ref{coro:result_come}}
In this section, we provide the  proof of our corollary for the sample complexity of $\SVRCp$. To prove Corollary \ref{coro:result_come}, we need following lemma:
\begin{lemma}\label{thm:2}
With the parameter choice in Corollary \ref{coro:result_come}, we further choose the parameters $\alphast, \betast$ in Theorem \ref{thm:1} as
\begin{align*}
    \alphast = n^{1/6}, \betast = n^{2/3}.
\end{align*}
From now, we can define $\Gst$ and $\cold$ as variables only dependent on $s, t, n, \glip$ and $\Hlip$. Then we have that $\Gst$ and $\cold$ are positive, and 
\begin{align}\label{estima_gamma}
    \constant_l\cdot\Hlip^{-1/2}<\Gst<\constant_u\cdot\Hlip^{-1/2},
\end{align}
where $\constant_l, \constant_u$ are two positive constants.
\end{lemma}




\begin{proof}[Proof of Corollary \ref{coro:result_come}] Since we already have $\Mst = \bigO(\Hlip)$ by the parameter choice in Lemma \ref{thm:2}, we only need to make sure that $\Delta_F / (ST\minG) \leq \epsilon^{3/2}$. Take $\minG > \constant_1\cdot \Hlip^{-1/2}$ and $T = n^{1/3}$, it is sufficient to let $S = [\constant_s\cdot(\Hlip^{1/2}\Delta_F)/(\epsilon^{3/2}n^{1/3})]+1$, where $\constant_s$ is a constant. Thus, as we need to sample $n$ Hessian at the beginning of each inner loop, and in each inner loop, we need to sample $\Bst = D_h = \tbigO(n^{2/3})$ Hessian, then the total sample complexity of Hessian for Algorithm \ref{algorithm:1} is $S\cdot n+ S\cdot T \cdot D_h = \tbigO(n+n^{2/3}\cdot (\Delta_F \sqrt{\Hlip})/\epsilon^{3/2})$.

\end{proof}

\section{Proof of the Key Lemmas}\label{proof_part1}
\subsection{Proof of Lemma \ref{lemma_oneloop}}
For simplification, we denote $\bg := D_g/4\cdot \Hlip^2/\glip^2$ and $\bH := D_H/\log d$. In this section, we prove the key lemma about the Lyapunov function \eqref{def:lya} used in the proof of our main theory. We define $\ev, \eU$ for the simplification of notation:
\begin{align}\label{def:eveu}
    \ev = \dF(\xold) - \vst,\quad \eU = \HF(\xold) - \Ust, 
\end{align}
where $\vb_0^s = \rg, \Ub_0^s = \rH$. Before we state the proof, we present some technical lemmas that are useful in our analysis. 

Firstly, we give a sharp bound of $\mu(\xnew)$. A very crucial observation is that we can bound the norm of gradient $\|\dF(\xnew)\|_2$ and the smallest eigenvalue of Hessian $\mineig(\HF(\xnew))$ with $\|\hst\|_2$, $\|\ev\|_2$ and $\|\eU\|_2$ defined in \eqref{def:eveu}. 
Formally, we have the following lemma:
\begin{lemma}\label{lemma:mu_evaluate}
Under the same assumption as in Theorem \ref{thm:1},
let $\hst, \xnew, \Mst$ be variables defined by Algorithm \ref{algorithm:1}. Then we have
\begin{align}\label{lemma:mu_evaluate_eq}
\mu(\xnew)/\constant_\mu &\leq 
 \Mst^{3/2}\|\hst\|_2^3  + \|\ev\|_2^{3/2} +   \Mst^{-3/2}\|\eU\|_2^3,
\end{align}
where $\constant_{\mu} = 18$.
\end{lemma}
Lemma \ref{lemma:mu_evaluate} suggests that to bound our target $\EE\mu(\xnew)$, we only need to focus on $\EE\|\hst\|_2^3, \EE\|\ev\|_2^{3/2}$ and $\EE\|\eU\|_2^3$.

Secondly, we bound $\EE\big[F(\xold)+\cold\|\xold - \rx\|_2^3\big] - \EE\big[F(\xnew)+\cnew\|\xnew - \rx\|_2^3\big]$. We first notice that $F(\xold) - F(\xnew)$ can be bounded with  $\ev, \eU$ and $\hst$. Such bound can be derived straightly from Hessian Lipschitz condition:

\begin{lemma}\label{lemma:fxnew_eva}
Under the same assumption as in Theorem \ref{thm:1},
let $\hst, \xold, \xnew, \Mst$ be variables defined by Algorithm \ref{algorithm:1}. Then we have the following result:
\begin{align}\label{lemma:fxnew_eva_eq}
    F(\xnew) &\leq F(\xold) - \Mst/12 \cdot \|\hst\|_2^3+ \constant_1\big(\|\ev\|_2^{3/2} /\Mst^{1/2} + \|\eU\|_2^3/\Mst^2\big), 
\end{align}
where $\constant_1 = 200$. 
\end{lemma}

We also give following result to show how to bound $\|\xnew - \rx\|_2^3$ with $\|\xold - \rx\|_2^3$:
\begin{lemma}\label{lemma:xnew_eva}
Under the same assumption as in Theorem \ref{thm:1},
let $\hst, \xold, \xnew, \Mst$ be variables defined by Algorithm \ref{algorithm:1}, $\alphast, \betast$ are parameters defined in Theorem \ref{thm:1}, then we have the following result:
\begin{align}\label{lemma:xnew_eva_eq}
    &\|\xnew - \rx\|_2^3 \leq (1+2\alphast+\betast)\|\hst\|_2^3+(1+1/\alphast^2+2/\betast^{1/2})\|\xold - \rx\|_2^3.
\end{align}
\end{lemma}

Based on Lemmas \ref{lemma:mu_evaluate},    \ref{lemma:fxnew_eva} and  \ref{lemma:xnew_eva}, we have established the connection between $\mu(\xnew), F(\xold)+\cold\cdot\|\xold - \rx\|_2^3$ and $F(\xnew)+ \cnew \cdot \|\xnew - \rx\|_2^3$ with only $\|\ev\|_2, \|\eU\|_2$ and $\|\hst\|_2$.

Finally, we bound $\EE\|\ev\|_2^{3/2}$ and $\EE\|\eU\|_2^3$ with consequent vector and matrix concentration inequalities. Previous analysis of variance-reduced method in nonconvex setting for first-order method which only focus on the upper bound of variance of $\ev$ gives $\EE\|\ev\|_2^2$ an upper bound only associated with $\|\xold - \rx\|_2$, which guarantees the variance reduction \citep{Reddi2016Stochastic, allen2016variance}. In our proof, we also need to bound the variance for stochastic Hessian $\EE\|\eU\|_2^3$. Thus we have following two lemmas:
\begin{lemma}\label{lemma:grad_reduce_1}
Under the same assumption as in Theorem \ref{thm:1},
let $\xold, \vst$ and $\rx$ be the iterates defined in Algorithm \ref{algorithm:1}, $\bg$ is the parameter of batch size defined in Theorem \ref{thm:1}. Then we have
\begin{align*}
     \EE_{\vst} \|\ev\|_2^{{3/2}} \leq  \frac{\Hlip^{{3/2}}}{b^{{3/4}}}\|\xold - \rx\|_2^3,
\end{align*}
where $\EE_{\vst}$ only takes expectation over $\vst$.
\end{lemma}

\begin{lemma}\label{lemma:hess_reduce_1}
Under the same assumption as in Theorem \ref{thm:1},
let $\xold, \Ust$ and $\rx$ be iterates defined in Algorithm \ref{algorithm:1}, $\bH$ is the batch size defined in Theorem \ref{thm:1}, $\bH>25$. Then we have
\begin{align*}
     \EE_{\Ust} \|\eU\|_2^{{3/2}} \leq  \constant_h \cdot \frac{\Hlip^3}{\bH^{3/2}}\|\xold - \rx\|_2^3,
\end{align*}
where $\EE_{\Ust}$ only takes expectation over $\Ust$, $\constant_h = 15000$.
\end{lemma}
Lemmas \ref{lemma:grad_reduce_1} and \ref{lemma:hess_reduce_1} suggest that with carefully selection of batch size, both $\EE\|\ev\|_2^{3/2}$ and $\EE\|\eU\|_2^3$ can be bounded by $\EE\|\xold - \rx\|_2^3$.

Now we are ready to prove Lemma \ref{lemma_oneloop}.
\begin{proof}[Proof of Lemma \ref{lemma_oneloop}]
First we combine \eqref{lemma:fxnew_eva_eq} and \eqref{lemma:xnew_eva_eq} to get a bound of $F(\xnew)+\cnew\|\xnew - \rx\|_2^3$. Let \eqref{lemma:fxnew_eva_eq}$+$ $\cnew \times$ \eqref{lemma:xnew_eva_eq}, then we have
\begin{align}\label{main1_1}
    &F(\xnew)+\cnew\|\xnew - \rx\|_2^3 \notag\\ 
    &\leq F(\xold) - \|\hst\|_2^3\big[\Mst/12 - \cnew(1+2\alphast+\betast)\big]+\constant_1\big(\|\ev\|_2^{3/2} /\Mst^{1/2} + \|\eU\|_2^3/\Mst^2\big)\notag\\
    &\quad\quad+\cnew(1+1/\alphast^2+1/\betast^{1/2})\|\xold - \rx\|_2^3.
\end{align}
 Note that $\Gst = [\Mst/12 - \cnew(1+2\alphast+\betast)]/\Mst^{3/2}>0$, then we have following inequalities, which equal \eqref{lemma:mu_evaluate_eq}$\times \Gst+$\eqref{main1_1}:
\begin{align}\label{main1_2}
    &F(\xnew)+\cnew\|\xnew - \rx\|_2^3+\Gst\cdot \mu(\xnew)/\constant_\mu \notag \\
    & \leq
   \|\ev\|_2^{3/2}\cdot(\constant_1 /\Mst^{1/2}+\Gst) + \|\eU\|_2^3\cdot(\constant_1/\Mst^2+\Gst/\Mst^{3/2})\notag \\
   &\quad\quad +F(\xold)+\cnew\big(1+1/\alphast^2+2/\betast^{1/2}\big)\|\xold - \rx\|_2^3.
\end{align}
Next we take total expectation on \eqref{main1_2} firstly, and use Lemmas \ref{lemma:grad_reduce_1} and \ref{lemma:hess_reduce_1} to bound $\EE\|\ev\|_2^{3/2}$ and $\EE\|\eU\|_2^3$, where 
\begin{align*}
    &\EE \|\ev\|_2^{{3/2}} \leq  \frac{\Hlip^{{3/2}}}{\bg^{{3/4}}}\EE\|\xold - \rx\|_2^3,\\
    &\EE\|\eU\|_2^3 \leq  \constant_h \cdot \frac{\Hlip^3}{\bH^{3/2}}\EE\|\xold - \rx\|_2^3.
\end{align*}
Then we get following results:
\begin{align}\label{main1_3}
     &\EE\bigg[F(\xnew)+\cnew\|\xnew - \rx\|_2^3+\Gst\cdot \mu(\xnew)/\constant_\mu \bigg]\notag \\
     &\leq 
     \EE\bigg[F(\xold)+\|\xold - \rx\|_2^3 \cdot\big[\cnew\big(1+1/\alphast^2+2/\betast^{1/2}\big)\notag \\
     &\quad\quad
     +(\constant_1 /\Mst^{1/2}+\Gst)\cdot \Hlip^{3/2}/\bg^{3/4}+\constant_h(\constant_1/\Mst^2+\Gst/\Mst^{3/2})\cdot \Hlip^3/\bH^{3/2}\big]\bigg].
\end{align}
By the definition of $\cold$, \eqref{main1_3} equals the following inequality, which is our result:
\begin{align}
    \Gst/\constant_\mu\cdot \EE(\mu(\xnew)) \leq \EE\bigg[F(\xold)+\cold\cdot \|\xold - \rx\|_2^3\bigg] - \EE\bigg[F(\xnew)+\cnew\cdot \|\xnew - \rx\|_2^3\bigg].
\end{align}
\end{proof}

\subsection{Proof of Lemma \ref{thm:2}}
\begin{proof}[Proof of Lemma \ref{thm:2}]
From the choice of parameters, we have $\Gst = \Gamma_{s',t}$ and $\cold = c_{s',t}$ for any pairs of $(s',s)$, thus for simplification, we define $\Gt = \Gst$ and $\ct = \cold$. Then, the induction equalities \eqref{def_gst} and \eqref{def_cst} can be re-written as the following if we submit $\alphast, \betast, \bg, \bH$ and $\Mst$ into it: 
\begin{align}
    &\Gt = \big[\Hlip\cdot \constant_m/12 - \ctnew(1+2n^{1/6}+n^{2/3})\big]/(\constant_m\Hlip)^{3/2}, \label{def:gt}\\
    & \ct = \ctnew(1+3/n^{1/3})+\big[\constant_1/(\constant_m\Hlip)^{1/2}+\Gt\big]\frac{\Hlip^{3/2}}{n}+\big[\constant_1/(\constant_m\Hlip)^2+\Gt/(\constant_m\Hlip)^{3/2}\big]\frac{\Hlip^3}{n}.\label{def:ct}
\end{align}
Next we submit \eqref{def:gt} into \eqref{def:ct} and re-arrange it, we have:
\begin{align}
    \ct = \ctnew\cdot\bigg(1+\frac{\constant_a}{n^{1/3}}-\frac{\constant_b}{n^{5/6}}-\frac{\constant_c}{n}\bigg)+\constant_d\cdot\frac{\Hlip}{n}, 
\end{align}
where $\constant_a = 3-\constant_m^{-3/2} - \constant_m^{-3}, \constant_b = 2\constant_m^{-3/2}+2\constant_m^{-3}, \constant_c = \constant_m^{-3/2}+\constant_m^{-3}, \constant_d = (\constant_1+1/12)(\constant_m^{-2}+\constant_m^{-1/2})$ are all constants. We can check that 
\begin{align*}
    \frac{3}{n^{1/3}}>\frac{\constant_a}{n^{1/3}}-\frac{\constant_b}{n^{5/6}}-\frac{\constant_c}{n}>0, \constant_d>0,
\end{align*}
when $n>10$, thus by induction and $c_T = 0$, we have following results:
\begin{align}
    0 = c_T\leq\ct\leq c_0<\constant_d\cdot\big[(1+3/n^{1/3})^T-1\big]\cdot n^{-2/3}\Hlip<30\constant_d \cdot n^{-2/3}\Hlip, \label{ceva}
\end{align}
where the last inequality holds because $(1+3n^{-1/3})^{n^{1/3}}<30$. Substituting \eqref{ceva} into \eqref{def:gt} and use the fact that $n>10$, we get our final result:
\begin{align}
    \constant_l\cdot\Hlip^{-1/2}<\Gst<\constant_u\cdot\Hlip^{-1/2},
\end{align}
where $\constant_l = (\constant_m/12 - 60\constant_d)/\constant_m^{3/2}>0, \constant_u = \constant_m^{-1/2}/12>0$ are constants.

\end{proof}

\section{Proof of Technical Lemmas}\label{proof_part2}
In this section, we prove the technical lemmas used in the proof of Lemma \ref{lemma_oneloop}.


\subsection{Proof of Lemma \ref{lemma:mu_evaluate}}

In order to prove Lemma \ref{lemma:mu_evaluate}, we need to the following two lemmas.

\begin{lemma}\citep{zhou2018stochastic}\label{lemma:grad_evaluate}
Under Assumption \ref{assumption:hess_lip}, for any $\hb \in \RR^d$, we have
\begin{align*}
    \|\dF(\xold + \hb)\|_2
    & \leq 
    \frac{\Hlip+2\Mst}{2}\|\hb\|_2^2
    +
    \|\nabla m_t^s(\hb)\|_2+ \|\dF(\xold) - \vst\|_2\\
    &\qquad+  \frac{1}{2\Mst} \|\nabla^2 F(\xold) - \Ust\|_2^2.
\end{align*}
\end{lemma}


\begin{lemma}\citep{zhou2018stochastic}\label{lemma:hess_evaluate}
Under Assumption \ref{assumption:hess_lip}, for any $\hb \in \RR^d$, we have
$$ -\lambda_{\min} \big(\nabla^2 F(\xold+\hb)\big) \leq  \frac{\Mst}{2}\|\hst\|_2 + \|\nabla^2 F(\xold) - \Ust\|_2+  \Hlip\|\hb\|_2.$$
\end{lemma}

\begin{proof}[Proof of Lemma \ref{lemma:mu_evaluate}]
Note that $\xnew = \xold+\hst$, so we use Lemma \ref{lemma:grad_evaluate} and Lemma \ref{lemma:hess_evaluate} to bound $\mu(\xnew)$, where we set $\hb = \hst$. To bound $\|\dF(\xnew)\|_2^{{3/2}}$, we apply Lemma \ref{lemma:grad_evaluate}:
\begin{align*}
   & \|\dF(\xnew)\|_2^{{3/2}}\\
   & \leq
   \Big[ 
    \frac{\Hlip+2\Mst}{2}\|\hst\|_2^2
    +
    \|\nabla m_t^s(\hst)\|_2+ \|\dF(\xold) - \vst\|_2 +  \frac{1}{2\Mst} \|\nabla^2 F(\xold) - \Ust\|_2^2 \Big]^{{3/2}}\\
   & \leq 2\bigg[\bigg(\frac{\Hlip+2\Mst}{2}\bigg)^{{3/2}}\|\hst\|_2^3 + \|\dF(\xold) - \vst\|_2^{{3/2}} + \|\nabla m_t^s(\hst)\|_2^{3/2}\\
   &\qquad+ \bigg(\frac{1}{2\Mst}\bigg)^{{3/2}} \|\nabla^2 F(\xold) - \Ust\|_2^3\bigg] \\
   & \leq
   18\Big[\Mst^{{3/2}}\|\hst\|_2^3 + \|\dF(\xold) - \vst\|_2^{{3/2}} +   \Mst^{-{3/2}}\|\nabla^2 F(\xold) - \Ust\|_2^3\Big],
\end{align*}
where the second inequality holds due to the following basic inequality
\begin{align*}
(a+b+c+d)^{{3/2}} \leq 2(a^{{3/2}}+b^{{3/2}}+c^{{3/2}}+d^{3/2}),
\end{align*}
and in the last inequality we use the assumption that $\Hlip<\Mst$ and the fact that $2\cdot(3/2)^{3/2}\leq 18$ and $\nabla m_t^s(\hst) = 0$. Next we bound $-\Mst^{-{3/2}}\big[\lambda_{\min} \big(\nabla^2 F(\xnew)\big)\big]^3$. By Lemma \ref{lemma:hess_reduce_1}, we have
\begin{align*}
   -\Mst^{-{3/2}}\big[\lambda_{\min} \big(\nabla^2 F(\xnew)\big)\big]^3 
   &\leq \Mst^{-{3/2}}\Big[\frac{\Mst}{2}\|\hst\|_2 + \|\nabla^2 F(\xold) - \Ust\|_2+  \Hlip\|\hst\|_2\Big]^3 \\
    & \leq 
    9\Mst^{-{3/2}}\cdot\bigg[\frac{\Mst^3}{8}\|\hst\|_2^3 + \|\nabla^2 F(\xold) - \Ust\|_2^3+ \Mst^3 \|\hst\|_2^3\bigg] \\
    & \leq
    9\Big[2\Mst^{{3/2}}\|\hst\|_2^3 +  \Mst^{-{3/2}}\|\nabla^2 F(\xold) - \Ust\|_2^3 \Big],
\end{align*}
where the second inequality holds due to $(a+b+c)^3 \leq 9(a^3+b^3+c^3)$. Thus, we have 
\begin{align*}
\mu(\xnew)& = \max \Big\{\|\dF(\xnew)\|_2^{{3/2}},  -\Mst^{-{3/2}}\big[\lambda_{\min} \big(\nabla^2 F(\xnew)\big)\big]^3\Big\} \\
& \leq 
 18\Big[\Mst^{{3/2}}\|\hst\|_2^3 + \|\dF(\xold) - \vst\|_2^{{3/2}}+   \Mst^{-{3/2}}\|\nabla^2 F(\xold) - \Ust\|_2^3\Big],
\end{align*}
which completes the proof. 
\end{proof}

\subsection{Proof of Lemma \ref{lemma:fxnew_eva}}
We have the following lemma about the sub-problem in cubic regularization:
\begin{lemma}\label{lemma:cubic_opt}\citep{zhou2018stochastic}
Let $\vst, \Ust, \hst, \Mst$ be variables defined by Algorithm \ref{algorithm:1}, then we have following basic results:
\begin{align*}
  &\vst +\Ust \hst + \frac{\Mst}{2}\|   \hst \|_2 \hst = 0,\\
   &\Ust+\frac{\Mst}{2}\|\hst\|_2\Ib \succeq 0,\\
   &\la\vst,\hst\ra + \frac{1}{2}\la\Ust\hst,\hst\ra+\frac{\Mst}{6}\|\hst\|_2^3 \leq -\frac{\Mst}{12}\|\hst\|_2^3.
\end{align*}
\end{lemma}

\begin{proof}[Proof of Lemma \ref{lemma:fxnew_eva}]
Note that $\xnew = \xold + \hst$, thus we apply Corollary \ref{assumption:hess_lip_cor}, and we have:
\begin{align}\label{main_2_0}
F(\xnew)& = F(\xold + \hst) \leq F(\xold) + \la \dF(\xold), \hst \ra + \frac{1}{2}\la \HF(\xold)\hst, \hst \ra + \frac{\Hlip}{6}\|\hst\|_2^3 \notag \\
& = 
 F(\xold) + \la \vst, \hst \ra + \frac{1}{2}\la \Ust\hst, \hst \ra + \frac{\Mst}{6}\|\hst\|_2^3 + \la \ev, \hst \ra \notag\\
 &\qquad+ \frac{1}{2}\la \eU\hst, \hst \ra+\frac{\Hlip-\Mst}{6}\|\hst\|_2^3.
\end{align}
From Lemma \ref{lemma:cubic_opt} we have
\begin{align}\label{main_2_1}
    \la \vst, \hst \ra + \frac{1}{2}\la \Ust\hst, \hst \ra + \frac{\Mst}{6}\|\hst\|_2^3 \leq \frac{-\Mst}{12}\|\hst\|_2^3.
\end{align}
Meanwhile, we have following two bounds on $\la \ev, \hst \ra$ and $\la \eU\hst, \hst \ra$ by Young's inequality:
\begin{align}
    &\la \ev, \hst \ra \leq \constant_4\|\ev\|_2\cdot \frac{1}{\constant_4}\|\hst\|_2 \leq \constant_4^{3/2}\|\ev\|_2^{3/2}+\frac{1}{\constant_4^3}\|\hst\|_2^3,\label{main_2_2}\\
    &\la \eU\hst, \hst \ra \leq \constant_5^2\|\eU\|_2\cdot\bigg(\frac{\|\hst\|_2}{\constant_5}\bigg)^2 \leq \constant_5^6\|\eU\|_2^3 + \frac{\|\hst\|_2^3}{\constant_5^3}.\label{main_2_3}
\end{align}
We set $\constant_4 = \constant_5 = (18/\Mst)^{1/3}$. Finally, because $\Hlip \leq \Mst/2$, we have
\begin{align}\label{main_2_4}
    \frac{\Hlip-\Mst}{6}\|\hst\|_2^3 \leq \frac{-\Mst}{12}\|\hst\|_2^3.
\end{align}
Substituting \eqref{main_2_1}, \eqref{main_2_2}, \eqref{main_2_3} and \eqref{main_2_4} into \eqref{main_2_0}, we have the final result:
\begin{align}
    F(\xnew) \leq F(\xold) - \frac{\Mst}{12}\|\hst\|_2^3 + \constant_1\bigg(\frac{\|\ev\|_2^{3/2}}{\Mst^{1/2}}+\frac{\|\eU\|_2^3}{\Mst^2}\bigg), 
\end{align}
where $\constant_1 = 200$.
\end{proof}

\subsection{Proof of Lemma \ref{lemma:xnew_eva}}
\begin{proof}[Proof of Lemma \ref{lemma:xnew_eva}]
Note that $\xnew = \xold+\hst$, then we have
\begin{align}\label{main2_8}
    \|\xnew - \rx\|_2^3 
    & \leq \big(\|\xold-\rx\|_2+\|\hst\|_2\big)^3 \notag\\
    & = \|\xold-\rx\|_2^3+\|\hst\|_2^3+3\|\xold-\rx\|_2^2\cdot\|\hst\|_2+3\|\xold-\rx\|_2\cdot\|\hst\|_2^2. 
\end{align}
The inequality holds due to triangle inequality. Next, we bound $\|\xold-\rx\|_2^2\cdot\|\hst\|_2$ and $\|\xold-\rx\|_2\cdot\|\hst\|_2^2$ by Young's inequality:
\begin{align}
    &\|\xold-\rx\|_2^2\cdot\|\hst\|_2 \leq \|\xold - \rx\|_2^3\cdot \frac{2}{3\betast^{1/2}}+\|\hst\|_2^3\cdot \frac{\betast}{3}, \label{main2_6}\\
    &\|\xold-\rx\|_2\cdot\|\hst\|_2^2 \leq \|\xold - \rx\|_2^3\cdot\frac{1}{3\alphast^2}+\|\hst\|_2^3\cdot \frac{2\alphast}{3}\label{main2_7}.
\end{align}
Substituting \eqref{main2_6}, \eqref{main2_7} into \eqref{main2_8}, we have the result:
\begin{align*}
    \|\xnew - \rx\|_2^3 \leq \|\xold - \rx\|_2^3\cdot(1+1/\alphast^2+2/\betast^{1/2})+\|\hst\|_2^3\cdot(1+2\alphast+\betast).
\end{align*}
\end{proof}
\subsection{Proof of Lemma \ref{lemma:grad_reduce_1}}
First we have the following lemma:
\begin{lemma}\citep{zhou2018stochastic}\label{lemma:moment_reduce_grad}
(Moment reduce)Suppose that $0<\alpha <2$, $\ab_1, \dots, \ab_N$ iid, $\EE \ab_i = 0$, then
\begin{align*}
    \EE\bigg\|\frac{1}{N} \sum_{i=1}^N \ab_i\bigg\|_2^{\alpha} \leq \frac{1}{N^{\alpha/2}}\big( \EE \|\ab_i\|_2^2 \big)^{\alpha/2}.
\end{align*}
\end{lemma}

\begin{proof}[Proof of Lemma \ref{lemma:grad_reduce_1}]
For simplicity, we use $\EE$ to replace $\EE_{\vst}$. First, to rearrange $\vst - \dF(\xold)$, we have
\begin{align*}
    \EE \|\dF(\xold) - \vst\|_2^{{3/2}}  = \EE\bigg\|\frac{1}{\bst} \sum_{i_t \in I_g} \bigg[\df _{i_t}(\xold) -  \df _{i_t}(\rx) +  \rg - 
    \dF(\xold)\bigg]\bigg\|_2^{3/2}.
\end{align*}
Now we use Lemma \ref{lemma:moment_reduce_grad}.  We set 
$\alpha  = 3/2,N=\bst$, and $\ab_i = \df _{i_t}(\xold) -  \df _{i_t}(\rx) -  \dF(\xold) +  \rg$,
we can check that $\alpha, N, \ab_i$ satisfy the assumption of Lemma \ref{lemma:moment_reduce_grad}. Thus, by Lemma  \ref{lemma:moment_reduce_grad}, we have 
\begin{align}\label{grad_reduce_2}
     \EE \|\dF(\xold) - \vst\|_2^{3/2} 
     & \leq \frac{1}{\bst^{{3/4}}}\bigg( \EE \|  \df _{i_t}(\xold) -  \df _{i_t}(\rx) +\rg -  \dF(\xold)\|_2^2\bigg)^{3/4}. 
\end{align}
Next we bound $\|  \df _{i_t}(\xold) -  \df _{i_t}(\rx) +\rg -  \dF(\xold)\|_2$. By Assumption \ref{assumption:grad_lip}, we have
\begin{align}\label{grad_reduce_1}
     &\| \df _{i_t}(\xold) -  \df _{i_t}(\rx) +\rg -  \dF(\xold)\|_2 
     \notag\\
     &\leq  \|\df _{i_t}(\xold) -  \df _{i_t}(\rx)\|_2 +\|\rg -  \dF(\xold)\|_2 \notag \\
     & =  \|\df _{i_t}(\xold) -  \df _{i_t}(\rx)\|_2 +\|\dF(\rx) -  \dF(\xold)\|_2\notag \\
     &\leq 2\glip\|\xold - \rx\|_2.
\end{align}
Finally, substituting \eqref{grad_reduce_1} and $\bst = b/\|\xold - \rx\|_2^2\cdot 4\glip^2/\Hlip^2$ into \eqref{grad_reduce_2}, we have 
\begin{align*}
     \EE \|\dF(\xold) - \vst\|_2^{3/2} 
     &\leq \frac{\Hlip^{3/2}}{(4\bg)^{3/4}\glip^{3/2}}\cdot    \|\xold - \rx\|_2^{3/2}\cdot
     \bigg(4\glip^2\|\xold - \rx\|_2^2\bigg)^{3/4} \notag \\
     &\leq \frac{\Hlip^{{3/2}}}{\bg^{3/4}}\|\xb_t^{s+1} - \hat{\xb}^s\|_2^3.
\end{align*}
\end{proof}

\subsection{Proof of Lemma \ref{lemma:hess_reduce_1}}
First we have the following lemma:
\begin{lemma}\citep{zhou2018stochastic}\label{lemma:moment_reduce_tropp}
Suppose that $q \geq 2, p \geq 2$, and fix $r \geq \max\{q, 2 \log p\}$. Consider $\Yb_1, ..., \Yb_N$ of i.i.d. random self-adjoint matrices with dimension $p \times p$, $\EE \Yb_i = \zero$, then
$$\bigg[\EE\bigg\|\sum_{i=1}^N\Yb_i\bigg\|_2^q\bigg]^{1/q}\leq 2\sqrt{er}\Big\|\Big(\sum_{i=1}^N\EE \Yb_i^2 \Big)^{1/2}\Big\|_2+4er\big(\EE \max_i \|\Yb_i\|_2 ^q\big)^{1/q}.$$
\end{lemma}

\begin{proof}[Proof of Lemma \ref{lemma:hess_reduce_1}]
We replace $\EE_{\Ust}$ with $\EE$ for simplification. We have
\begin{align}\label{eq:hess_reduce_1_2}
    \EE\|\nabla^2 F(\xold) - \Ust\|_2^3  &= \EE \bigg\|\nabla^2 F(\xold)-\frac{1}{\Bst}\bigg(\sum_{j_t \in I_h} \big(\Hf_{j_t}(\xold) - \Hf_{j_t}(\rx) +\rH\big)\bigg)\bigg\|_2^3 \notag\\
    & = \EE \bigg\|\frac{1}{\Bst}\bigg[\sum_{j_t \in I_h} \big[\Hf_{j_t}(\xold) - \Hf_{j_t}(\rx) +\rH -\nabla^2 F(\xold) \big]\bigg] \bigg\|_2^3. 
\end{align}
We use Lemma \ref{lemma:moment_reduce_tropp}, set
$$q=3,p=d, r = 2\log p , \Yb_i = \Hf_{j_t}(\xold) - \Hf_{j_t}(\rx) +\rH -\nabla^2 F(\xold),N=\Bst.$$
We can check that these parameters satisfy the Assumption \ref{lemma:moment_reduce_tropp}. Meanwhile, $\Yb_i$ are i.i.d random variables, and by Assumption \ref{assumption:hess_lip}, we have
\begin{align}\label{eq:hess_reduce_1_3}
\|\Yb_i\|_2 &=\Big\|\Hf_{j_t}(\xold) - \Hf_{j_t}(\rx) +\rH -\nabla^2 F(\xold)  \Big\|_2 \notag\\
 &\leq 
 \Big\|\Hf_{j_t}(\xold) - \Hf_{j_t}(\rx) \Big\|_2 + \Big\|\rH -\nabla^2 F(\xold)  \Big\|_2 \notag\\
 &\leq  
 \Hlip \|\xold - \rx\|_2 +  \Hlip \|\xold - \rx\|_2 \notag \\
 &= 2\Hlip \|\xold - \rx\|_2 .
\end{align} 
By Lemma \ref{lemma:moment_reduce_tropp}, we have
\begin{align}\label{eq:hess_reduce_1_5}
   \Big[\EE\bigg\|\sum_{i=1}^{\Bst} \Yb_i\bigg\|_2^3\Big]^{1/3}\leq 2\sqrt{er}\Big\|\Big(\sum_{i=1}^{\Bst} \EE \Yb_i^2 \Big)^{1/2}\Big\|_2+4er\big(\EE \max_i \|\Yb_i\|_2 ^3\big)^{1/3}. 
\end{align}
Next we give the upper bounds of \eqref{eq:hess_reduce_1_5}. The first term in RHS of \eqref{eq:hess_reduce_1_5} can be bounded as
\begin{align}\label{eq:hess_reduce_1_6}
    2\sqrt{er}\Big\|\Big(\sum_{i=1}^{\Bst}\EE \Yb_i^2 \Big)^{1/2}\Big\|_2 &= 2\sqrt{er}\Big\|\sum_{i=1}^{\Bst}\EE \Yb_i^2\Big\|_2 ^{1/2} \notag \\
    &=  2\sqrt{\Bst er}\Big\|\EE \Yb_i^2\Big\|_2 ^{1/2} \notag\\
    & \leq 
     2\sqrt{\Bst er}\Big(\EE \|\Yb_i\|^2\Big)_2 ^{1/2}\notag \\
     &\leq 
     4\Hlip\sqrt{\Bst er} \|\xold - \rx\|_2 ,
\end{align}
where the last inequality holds due to \eqref{eq:hess_reduce_1_3}.
The second term in RHS of \eqref{eq:hess_reduce_1_5} can be bounded as
\begin{align}\label{eq:hess_reduce_1_7}
   4er\big(\EE \max_i \|\Yb_i\|_2 ^3\big)^{1/3} &\leq 4er[(2\Hlip \|\xold - \rx\|_2)^3]^{1/3}\notag \\
   &= 8\Hlip er  \|\xold - \rx\|_2.
\end{align}
Substituting \eqref{eq:hess_reduce_1_6}, \eqref{eq:hess_reduce_1_7} into \eqref{eq:hess_reduce_1_5}, we have
\begin{align*}\label{eq:hess_reduce_1_8}
   \Big[\EE\bigg\|\sum_{i=1}^{\Bst}\Yb_i\bigg\|_2^3\Big]^{1/3}&\leq 4\Hlip\sqrt{\Bst er} \|\xold - \rx\|_2+8\Hlip er  \|\xold - \rx\|_2,
\end{align*}
which equals
\begin{align}
 \EE\bigg\|\frac{1}{\Bst}\sum_{i=1}^{\Bst}\Yb_i\bigg\|_2^3 & \leq  64 \Hlip^3\bigg(\sqrt{\frac{er}{\Bst}}+ \frac{2er}{\Bst}\bigg)^3\|\xold - \rx\|_2^3.
\end{align}
Substituting \eqref{eq:hess_reduce_1_8} into \eqref{eq:hess_reduce_1_2}, we have
\begin{align*}
     \EE\|\nabla^2 F(\xold) - \Ust\|_2^3  &\leq 64 \Hlip^3\bigg(\sqrt{\frac{2e\log d}{\Bst}}+ \frac{4e\log d}{\Bst}\bigg)^3\|\xold - \rx\|_2^3 \\
     & = 
     64 \Hlip^3\bigg(\sqrt{\frac{2e}{\bH}}+ \frac{4e}{\bH}\bigg)^3\|\xold - \rx\|_2^3 \\
     &\leq 15000 \cdot \frac{\Hlip^3}{\bH^{3/2}}\|\xold - \rx\|_2^3,
\end{align*}
where the last inequality holds due to 
\begin{align*}
  64\bigg(\sqrt{\frac{2e}{\bH}}+ \frac{4e}{\bH}\bigg)^3 \leq \frac{15000}{\bH^{3/2}},
\end{align*}
when $\bH>25$.
\end{proof}

\section{Additional Lemmas}
We have the following equivalent definition for Hessian Lipschitz, where the proof can be found in \citet{Nesterov2006Cubic}:
\begin{corollary}\label{assumption:hess_lip_cor}
Suppose $F$ is a $\Hlip$ Hessian Lipschitz function, then we have following results:
\begin{align*}
    \|\HF(\xb) - \HF (\yb)\| &\leq \Hlip\|\xb - \yb\|_2 ,\\
     F(\xb+\hb) &\leq F(\xb)+ \la \dF(\xb), \hb \ra + \frac{1}{2}\la \HF(\xb)\hb, \hb \ra + \frac{\Hlip}{6}\|\hb\|_2^3, \\
     \|\dF(\xb+\hb) - \dF(\xb) - \HF(\xb)\hb\|_2 &\leq \frac{\Hlip}{2}\|\hb\|_2^2.
\end{align*}
\end{corollary}